\documentclass[12pt,a4paper,twoside]{amsart}
\usepackage{amsmath}
\usepackage{amssymb}
\usepackage{amsfonts}
\usepackage{a4}

\usepackage{tikz}
\tikzset{node distance=3cm, auto}

\numberwithin{equation}{section}

\theoremstyle{plain}

\newtheorem{theorem}[subsection]{Theorem}
\newtheorem{definition}[subsection]{Definition}
\newtheorem{proposition}[subsection]{Proposition}
\newtheorem{lemma}[subsection]{Lemma}
\newtheorem{remark}[subsection]{Remark}
\newtheorem{questioaffine formsn}[subsection]{Question}
\newtheorem{corollary}[subsection]{Corollary}

\newtheorem{claim}[subsection]{Claim}

\theoremstyle{definition}

\newcommand{\2}{{\bf 2}}

\newcommand{\mC}{{\mathbb C}}

\newcommand{\mE}{{\mathbb E}}
\newcommand{\mF}{\mathbb F}

\newcommand{\mN}{\mathbb N}

\newcommand{\mP}{\mathbb P}

\newcommand{\mZ}{{\mathbb Z}}

\newcommand{\ho}{\hookrightarrow}

\newcommand{\bo}{\omega}

\newcommand{\ep}{\epsilon}

\newcommand{\mcL}{\mathcal L}

\newcommand{\ti}{\tilde}

\newcommand{\emp}{\emptyset}

\usepackage{color}

\begin{document}

\title{Polynomial functions as splines. }
\begin{abstract}

Let $V$ be a vector space over a finite  field $k$. We give a condition on a subset  $A \subset V$ that allows for a local criterion for   
checking when a function $f:A \to k$ is a restriction of a polynomial function of degree $<m$ on $V$. In particular, we show that  high {\em rank} hypersurfaces of $V$  of degree $\ge m$ satisfy this condition.  In addition we show that  the criterion is robust (namely locally testable in the theoretical computer science jargon).

\end{abstract}

\author{David Kazhdan}
\address{Einstein Institute of Mathematics,
Edmond J. Safra Campus, Givaat Ram 
The Hebrew University of Jerusalem,
Jerusalem, 91904, Israel}
\email{david.kazhdan@mail.huji.ac.il}

\author{Tamar Ziegler}
\address{Einstein Institute of Mathematics,
Edmond J. Safra Campus, Givaat Ram 
The Hebrew University of Jerusalem,
Jerusalem, 91904, Israel}
\email{tamarz@math.huji.ac.il}

%\thanks{The second author is supported by ERC grant ErgComNum 682150. Part of the material in this paper is based upon work  supported by the National Science Foundation under Grant No. DMS-1440140 while the second author was in residence at the Mathematical Sciences Research Institute in Berkeley, California, during the Spring 2017 semester.}

\maketitle

\section{introduction}
Let $V$ be a vector space over a field $k$. A classical fact is that if $k$ is a prime field, a function  $f:V \to k$ is a polynomial of degree $< m$ if and only if it "vanishes on $m$-dimensional cubes", namely for all $x, h_1, \ldots, h_m \in V$ we have,
\[
(*)_f \quad \sum_{\omega \in \{0,1\}^m}(-1)^{|\omega|}f(x+ \sum_i \omega_i h_i )  =0,
\]
where $|\omega| = \sum_{i=1}^m \omega_i$.  
For example, a function is linear if and only if it vanishes on $2$-dimensional cubes, namely for all  $x, h_1, h_2 \in V$ we have, 
\[
f(x)-f(x+h_1)-f(x+h_2)+f(x+h_1+h_2)=0. 
\]
As is shown in  \cite{akklr, jprz, KR} this characterization is robust -  namely 
there exists a constant $C(m)$, such that any function $f$ for which $(*)_f$ holds for all but an $\epsilon < \epsilon(m)$ proportion of the  $m$-dimensional cubes in $V$  is $C(m)\epsilon$-close to a polynomial of degree $<m$.   In the theoretical computer science jargon this is referred to 
as - low degree polynomials on a vector space $V$ are {\it locally testable}.  The proof in  \cite{akklr} heavily relies on the existence of a group structure on $V$.  

In this paper we provide  a condition  on a subset  $A\subset  V$ which guarantees local testability of polynomiality.

Our first result is the existence of functions $C(m)$ and $\ep(m)$
such that  for any subset $A\subset V$ of given density and {\it (Gowers)$m$-uniformity} (see Definition \ref{uniform}), any function $f:A \to k$  such that $(*)_f$ holds for all but an $\ep \leq \epsilon (m)$ proportion of $m$-dimensional cubes in $A$,
there exists a polynomial $g:V \to k$ of degree $<m$ whose restriction to $A$ is equal to $f$ for all but $C(m) \epsilon$ proportion of points of $A$.

\begin{remark}This question is non trivial even in the case when $\epsilon=0$.  
\end{remark}
An important example (see \cite{gt, kl}) of  $m$-uniform subsets of $V$, are ones of the form $X(k)$ where $X$ is defined by  a  system  of  high {\em rank}  (see Definition \ref{alg} below) of homogeneous polynomials equations $\{P_i(v)=0\}_{i=1}^c$ where $P_i$ are of degrees  $\ge m$.

We also consider subsets of $V$ of the form  $X(k)$,  without the assumption that the degrees of the defining polynomials are $\ge m$. In the case when $X(k)$ is of  {\it high rank} (depending on the codimension $c$ and the degrees of the $P_i$), we show that the set $X(k)$ has the following
property:  for any function $f:X(k) \to k$ vanishing on all but an $\epsilon \leq \epsilon (m)$ proportion of $m$-dimensional cubes with {\em in  $X(k)$},
there exists a function $g:X(k) \to k$ which vanishes on {\em all}  $m$-dimensional cubes in $X(k)$, and  coincides with $f$ for all but $C(m)\epsilon $ proportion of points on $X(k)$.

\begin{remark}
 In \cite{kz-extension} we show by different methods the existence of a polynomial  $P$ on $V$ of degree $\leq m$ such that $P_{|X}=f$.
\end{remark}

In the case when $X(k)$ is not of high rank we can prove a weaker result, where the condition on the smallness of  $\epsilon$  depends on the finite field $k$ and on the codimension of $X$,. 

\begin{remark}
We prove analogous results for functions taking values in an arbitrary abelian group.
\end{remark}
 We expect our results to be useful for obtaining quantitative bounds for the inverse theorem for the $U_m$-Gowers  norms over finite fields \cite{btz,tz,tz-2}: we expect that  methods from additive combinatorics can be used to reduce the inverse theorem for the $U_m$-Gowers norms to questions of polynomial testing and polynomial extensions on high rank varieties.

\subsection{Definitions and and statement of results}
Let $V$ be a vector space over a field $k$.  An $m$-cube  in a vector space $V$ is a collection $(u|\bar v),u\in V,\bar v\in V^m$ of $2^m$ points 
$\{ u+\sum _{i=1}^m \bo _iv_i\}$, $\bo _i\in \{0 ,1\}$.

For any map  $f:V\to H$ where $H$ is an abelian group we denote by $f_m$ the map from the set $C_m(V)$ of $m$-cubes to $H$ given by 
$$f_m(u|\bar v)=\sum _{\bar \bo \in \{ 0 ,1\}^m}(-1)^{|w|}f(u+\sum _{i=1}^m\bo _iv_i)$$
where $|\bo|=\sum _{i=1}^m\bo _i$.  For a subset  $X \subset V$ we denote 
 $C_m(X)$ the set of $m$-cubes in $V$ with all vertices in $X$. Note that in the case that $H=k$, where $k$ a prime field, functions $f:V \to k$ such that 
 $f_m$ vanishes on $C_m(G)$ are precisely polynomials of degree $<m$.

\begin{definition}[Gowers norms \cite{gowers}]\label{uniform} For a function $g: V \to \mC$ we define the norm $\|g\|_{U_m} $ by
\[
\|g\|^{2^m}_{U_m} = \mE_{x,v_1, \ldots v_m\in V} \prod_{\omega \in \2^m} g^{\omega}(x+\omega \cdot \bar v),
\]
where $g^{\omega}=g$ 
if $|\omega|$ is even and $g^{\bo}=\bar g$ otherwise.  We say that $X \subset V$ is {\em $(\epsilon, m)$-uniform} if $\|1_X-\mE1_X\|_{U_m}<\epsilon$.  \\
\end{definition}

\begin{theorem}\label{extending-uniform}  For any 
$\delta>0$ there exists $\epsilon = \epsilon (\delta)$ such that for any $(\epsilon,m)$-uniform subset $X$ of $V$  of
 density $>\delta$   and a function $f:X \to H$   such that
 $f_m|_{C_m(X)} \equiv 0$,   there exists a function $h:V \to H$ with $h_m \equiv 0$ such that $h|_X=f|_X$.  Moreover we can take $\epsilon=\delta^{O_m(1)}$.
\end{theorem}

We say that a property $P$ is satisfied $\epsilon$-a.e. $x\in X$ if it is satisfied for $(1-\epsilon )|X|$ elements $x \in X$.

\begin{theorem}\label{testing-uniform}
Let $m \ge 1$.  There exist  $\alpha , B, C>0$ depending on $m$ such that the following holds:  For any $\delta>0$,  any $0<\epsilon< \alpha$, $\eta<(\epsilon \delta)^{B}$, any  $(\eta, m)$-uniform subset $X$ of  $V$ of density $\delta$  and any $f:X \to H$ with $f_m(c)=0$ for $\epsilon$-a.e.  $c \in C_m(X)$, there exists a  function $h:V \to H$  such that $h_m\equiv 0$, and  $h (x)=f(x)$ on $C \epsilon$ a.e. $x \in X$. \end{theorem}

\begin{remark} By the monotonicity of the Gowers norms, the Theorem \ref{testing-uniform} holds for any $f$ satisfying $f_d(c)=0$ for $\epsilon$-a.e.$c \in C_d(X)$ for any $d \le m$. When $m$ is much larger than $d$ ($>d2^d)$ the result can be obtained with a simpler argument. 
\end{remark}

 \begin{definition}[Rank]\label{alg}
\begin{enumerate}
 \item Let $P:V \to k$ be a polynomial of degree $d$. We define the rank $r^{d}(P)$ as the minimal number $r$ such that $P$ can be written as a sum $P=\sum _{j=1}^rQ_jR_j$ where $Q_j,R_j$ are polynomials of degees $<d$ defined over $k$. Often we write $r(P)$ instead of $r^d(P)$.
\item Let $\bar P=\{ P_i\}, 1\leq i\leq c $ be a  family of polynomials of degree $d$. We define $r(\bar P)$ as the minimal $d$-rank of non-trivial linear combinations of $P_i$.
\item Given any family $\bar P=\{ P_i\},1\leq i\leq c, \deg(P_i)\geq 2$ we write it as a disjoint union $\bar P =\bigcup _{j=2}^d \bar P ^j$ where $\bar P ^j$ is a family of polynomials of degree $j$. We define $r(\bar P):=\min _jr^{j}(\bar P ^j).$ 
\end{enumerate}
\end{definition}

Let $k$ be a finite field of size $q$. 
Let $V$ be a $k$-vector space and  $X$ a subvariety  of degree $d$  which is a complete intersection of  codimension $L$, with all  defining polynomials of degree $> m$. In the appendix we show that $X$ is of density $\ge q^{-O_{d,L}(1)}$, and by \cite{bl} (Theorem 4.8) for any $s>0$ there is $r=r(k,d,L)$ such that if the rank of $X$ is $>r$ then $X$ is $(q^{-s}, m)$-uniform. The following result is an application of Theorem \ref{extending-uniform}:

\begin{corollary}\label{high-rank} Let $k$ be a finite field, and let $d \ge m> 0, L>0$. There exists $r=r(k,d, L)>0$  such that for any $k$-vector space $V$ any   subvariety $X$   of rank $>r$, degree $d$  which is a complete intersection of  codimension $L$, with all  defining polynomials of degree $\ge m$ and any function  $f:X \to k$  such that
 $f_m|_{C_m(X)} \equiv 0$, there exists a function $h:V \to k$ with $h_m \equiv 0$ such that $h|_X=f|_X$. 
\end{corollary}

We also prove a splining result for subvarieties $X\subset V$ where $V$  is a finite-dimensional vector space over a finite field $k=\mF _q$, which is independent of rank. We use this result in \cite{kz-extension}. 

\begin{theorem}[Splining on $X$]\label{testing-X-intro}
For any $m,d, L>0$ there exists positive real numbers $A, B$ depending on $d,L,m$, such that the following holds: for any  complete intersection $X\subset V$   of degree $d$ codimension $L$,  any $0<\epsilon<q^{-A}$, and any function $f:X\to H$ such that 
$f_m$ vanishes $\epsilon$-a.e on $C_m(X)$ 
 there exists a function $h: X \to H$ such that $h_m|_{C_m(X)}\equiv  0$ and $h(x)=f(x)$ for 
$q^{B} \epsilon$ a.e $x \in X$. 
\end{theorem}

\begin{theorem}[Subspace splining on $X$]\label{testing-subspace-X-intro}
Let $m,d, L>0$. There exists an $A, B >0$ depending on $d,L,m$, such that the following holds: for any vector space $V$ over $k$, any complete intersection $X \subset V$ of degree $d$, codimension $L$ and  a function $f:X\to k$ such that  the restriction of $f$ to $\epsilon$-a.e  affine subspace of dimension $l=\lceil \frac{m}{q-q/p}\rceil$ is a polynomial of degree $<m$, where  $\epsilon < q^{-A}$,  there exists a function $h: X \to k$ such that the restriction of $h$ to any affine subspace of dimension $l$ is a polynomial of degree $<m$, and $h(x)=f(x)$ for
$q^{B} \epsilon$ a.e $x \in X$.
\end{theorem}

\begin{corollary} Let $m,d, L>0$.
There exists an $A=A(d,L,m) > 0$ such that the following holds: Let $X \subset V(k)$ be  a  complete intersection of degree $d$, codimension $L$. Then for any function $f : X \to  k$ such that the restriction of $f$ 
to any affine subspace of dimension  $\lceil \frac{m+1}{q-q/p}\rceil$ is a  polynomial of degree $m$, and the restriction of $f$ to $q^{-A}$ almost any  affine subspace 
of dimension $l = \lceil \frac{m}{q-q/p}\rceil$  is a polynomial of degree $< m$,  the restriction of $f$ 
to any affine subspace is a polynomial of degree $< m$.
\end{corollary}

In the high rank case we have a stronger result:

\begin{theorem}[Splining on $X$ high rank]\label{testing-X-high}
Fix $m>0$. There exist $\alpha, C>0$ (depending on $m$) such 
that  for any $\epsilon <\alpha$ for any $q,d,L$ there exists $r=r(q,d, L,m, \epsilon)>0$ such that  the following holds.
Let $V$ be an $\mF _q$-vector space, 
$X\subset V$  be a  subvariety which is a  complete intersection of  codimension $L$,  of degree $d$ and  rank $>r$ and let
 $f:X \to H$ be a map with $f_m(c)=0$ for $\epsilon$-a.e.  $c \in C_m(X)$. Then there exists a  function $h:X \to H$  such that $h_m\equiv 0$, and  $h (x)=f(x)$ on $C \epsilon$ a.e. $x \in X$. For $d>char(k)$, we can have $r=r(q,d, L,m,\epsilon)=r(d, L,m, \epsilon)$. 
 \end{theorem}

 In section \ref{follows}
we show that Theorem \ref{extending-uniform} follows from Theorem  \ref{testing-uniform}. Now
we describe the proof of Theorem \ref{testing-uniform} which follows the lines of the  proof of polynomial splining for vector spaces over finite fields \cite{akklr}.
 We show that for any given $x \in V$  the function $F_x(\bar v):=f_m(x|\bar v)-f(x)$ is constant for almost all 
$\bar v$ such that $(x|\bar v)' \in C_m(X)'$, where $C_m(X)'$ denotes the set of almost cubes in $X$ (see Definition \ref{def-cubes}),  
and  {\it almost all}  depends on the uniformity of $X$. While in the case that $X$ is a vector space (\cite{akklr}) this is straight forward, in the case when $X$ is a uniform set it becomes rather tricky, and the  key insight is that uniformity gives control of the sizes of fibers of various maps between subsets of $X^M$.  Using this  {\it almost constancy} we define a  function $h$ on $V$ as the {\it essential value} of $F_x(\bar v)$. Next we use   the uniformity of $X$  to show that $h_m$ vanishes on $C_m(V)$. Again, in the case when $X$ is a vector space this is straight forward, while in the case of uniform varieties much less so.     Finally we show  that $h=f$ a.e. on $X$. 

 Theorem \ref{testing-X-intro} is proved in a similar manner, but without of the assumption of the  uniformity of $X$  we can not extend $h$ from $X$ to $V$. 
 The key feature we use is the abundance of solutions to various systems of equations. We  derive the abundance from the uniformity in the  case of  Theorem  \ref{testing-uniform}, and from a general result about existence of many solutions for  some system of equations (see Proposition
  \ref{solutions-X}). Theorem  \ref{testing-X-high} is proved along the lines of Theorem  \ref{testing-uniform}. While we cannot extend $f$ to $V$ we can fix it within $X$ in a way where the bound on $\epsilon$ does not depend on the density using the fact that in the high rank case we still have good control on the  
sizes of fibers of various maps between subvarieties of $X^M$. This result is close in spirit to the results in \cite{gt, kl} where it is shown that if $P$ is a polynomial of degree $d$ and $\epsilon$-a.e  we have $P= \Gamma(Q_1, \ldots, Q_M)$, where 
$Q_1, \ldots Q_M$ is a high rank collection of polynomials of degrees $<d$, then if $\epsilon$ is sufficiently small then actually   $P \equiv \Gamma(Q_1, \ldots, Q_M)$.  %%%%%%%%%%%%%%%%%%%%%%%%%%%%%%%%%%%%

 \subsection*{Acknowledgement}
 The second author is supported by ERC grant ErgComNum 682150. Part of the material in this paper is based upon work  supported by the National Science Foundation under Grant No. DMS-1440140 while the second author was in residence at the Mathematical Sciences Research Institute in Berkeley, California, during the Spring 2017 semester.
We thank the anonymous referee for offering simplified proofs for some Lemmas in the paper. 

\section{Complexity of linear systems}\label{algebra-lemmas}

We start with some notations related to cubes: 
 \begin{definition}\label{def-cubes}
 For any  $m\in \mZ_{>0}$ we 
 define $\2^m :=\{0,1\}^m$. We define $| \bo | =\sum _1^m\bo _i ,\bo \in \2^m$ and say that $\bo$ is even (odd) if   $| \bo | $ is even (odd). 
 Let $V$ be a vector space over a finite field $k$.  For any $\bo \in \2^m, \bar v =(v_1, \ldots, v_m ) \in V^m$ we define 
$\omega \cdot \bar v :=\sum _1^m \bo _iv_i\in V$.

 For any  $u\in V,\bar v =( v_1, \ldots, v_m )\in V^m$ we denote by $\phi _{(u|\bar v )}:2^m\to V$ the map  given by 
$$\phi _{(u|\bar v )}(\bo):=u+\omega \cdot \bar v $$
and denote by  $(u|\bar v)\subset V$ the image of the $\phi _{(u|\bar v )}$
 and  by $(u|\bar v)'\subset V$ the image of the restriction of  $\phi _{(u|\bar v )}$ to $2^m \setminus \{ 0\}$. We say that  the subsets of $V$ of the form $(u|\bar v)$ are {\it $m$-cubes} and that  subsets 
$V$ of the form $(u|\bar v)'$ are {\it almost cubes}.

For a subset  $X \subset V$ we denote by
 $C_m(X)$ the set of $m$-cubes in $V$ with all vertices in $X$ and  $C'_m(X)$ the set of {\it almost cubes} in $V$ with all vertices in $X$.

Let $H$ be an abelian group.  For any  $H$-valued  function $f$ on $X$ we denote by $f_m$ the function on $C_m(X)$  defined by
\[
f_m(u|\bar v) =\sum_{\omega \in \2^m}(-1)^{|\omega|} f(u+\omega\cdot \bar v)
\]
and  by $f'_m$ the function on $C'_m(X)$
 defined by
\[
f'_m(u|\bar v) =\sum_{\omega \in \2^m \setminus \{ 0\}}(-1)^{|\omega|} f(u+\omega\cdot \bar v).
\]
 We say that $c \in C_m(X)$ is {\em good for $f$} if $f_m(c)=0$. 
Given a function $f:X\to H$ we write $f_m(X)=0$ if all $c\in C_m(X)$ are good for $f$.
\end{definition}

We will need to manipulate various systems of linear forms. 
The following is a notion of complexity of linear forms introduces in \cite{gt} (up to shifter index). 

  \begin{definition}[CS complexity \cite{gt}]
Let $\bar r =\{ r_i\}_{i \in I}$ be a family of affine maps $r_i: V^r \to V$ of the form  $r_i(\bar v) = \sum a_{ij}v_j+w_i, a_{ij}\in \mZ$. 
 Say that $\bar r$ is of {\em CS complexity $\le d$} at $j \in I$ if we can partition $I\setminus \{j\}$ to $ d$ sets so that
 $r_j$ is not in the affine span of any set.  Say that $\bar r$ is of {\em CS complexity $\le d$}  if it is of CS complexity $\le d$ at any $j\in I$. 
 If $\bar r$ is of complexity $\le d$ at $j$, we call a partition of $I\setminus \{j\}$ to $ d$ sets so that
 $r_j$ is not in the affine span of any set an {\em admissible $d$-partition}. 
\end{definition} 

 \begin{remark}
 a) The complexity of an affine system is the same as that of its linear part. b) The complexity of a sub collection of linear  forms is bounded by the complexity of the full collection. 
\end{remark}

\begin{proposition}[\cite{gt}] 
Let $\bar r =\{ r_i\}_{i \in I}$ be a family of affine maps $r_i: V^r \to V$,  $r_i(\bar v) = \sum a_{ij}v_j+w_i, a_{ij}\in \mZ$. 
If $\bar r$ is of {\em CS complexity $\le d$} then for any $f_i:V \to \mathbb C$, $\|f_i\|_{\infty} \le 1$, $i\in I$
 \[
| \mE_{\bar  v \in V^r}\prod_{i\in I} f_i(r_i(\bar v))| \le \|f_j\|_{U_d}.
 \]
\end{proposition}

\begin{lemma}\label{double} Let $S$ be a system of $r_1,\ldots,  r_k$, $s_1,\ldots,  s_m$ be  of linear forms  in $\bar x$ and let  $t_1,\ldots,  t_m$
be non zero  linear forms in $\bar y$. Suppose the system $T=\{r_1(\bar x),\ldots,  r_k(\bar x), s_1(\bar x)+t_1(\bar y),\ldots,  s_m(\bar x)+t_m(\bar y)\}$ 
is of complexity $\le d$, and the system $\{s_1(\bar x)+t_1(\bar y),\ldots,  s_m(\bar x)+t_m(\bar y)\}$ is of complexity $\le d-1$. Then the system 
\[
S=\{r_1(\bar x),\ldots,  r_k(\bar x), s_1(\bar x)+t_1(\bar y),\ldots,  s_m(\bar x)+t_m(\bar y), s_1(\bar x)+t_1(\bar y'),\ldots,  s_m(\bar x)+t_m(\bar y')\}
\]
of linear forms in $\bar x, \bar y, \bar y'$ is of complexity $\le d$.  
\end{lemma}

We will refer to this Lemma as the doubling lemma - we fix a collection of variables and the linear forms including them, and "double" the variables in all other forms.

\begin{proof} Fix a form in the system $T$. If it is one of the $r_i$ then take an admissible $d$ partition of $T \setminus \{r_i\}$ and adjoin $s_l(\bar x)+t_l(\bar y')$ to the set in the partition of  $s_l(\bar x)+t_l(\bar y)$. If the form is  $s_j(\bar x)+t_j(\bar y)$, take a $d-1$ partition of $\{ s_1(\bar x)+t_1(\bar y),\ldots,  s_m(\bar x)+t_m(\bar y)\}\setminus \{s_j(\bar x)+t_j(\bar y)\}$, and add a new set to the partition containing all the $r_i$. This is a $d$ partition for $T \setminus \{s_j(\bar x +t_j(\bar y)\}$. Now for any $i \ne j$ adjoin $s_i(\bar x)+t_i(\bar y')$ to the set in the partition of  $s_i(\bar x)+t_i(\bar y)$. Finally adjoin $s_j(\bar x)+t_j(\bar y')$ to the set containing the $r_i$. This is a good $d$ partition for $S \setminus \{s_j(\bar x +t_j(\bar y))i\}$. By symmetry in $\bar y, \bar y'$ the complexity at any form $s_j(\bar x)+t_j(\bar y')$ is $\le m$. 
\end{proof}

\begin{lemma}\label{cube-complexity}
The system of linear forms in $(x, \bar v)$ corresponding to the points on the cube
\[
(x|v_1, v_2, \ldots, v_m)
\]
is of complexity $\le m$.
\end{lemma}

\begin{proof}
The corresponding forms are $T=\{x+\omega \cdot \bar v\}_{\bo \in \2^m}$, and correspond to the vertices of an $m$ dimensional cube. By cube symmetry it suffices to show that 
the complexity at $x$ is $\le m$. We prove this by induction. Partition $T \setminus \{x\}$ to two sets $T_1=\{x+\omega \cdot (v_1, \ldots, v_{m-1})\}_{\bo \in \2^{m-1}\setminus \bar 0}$, $ T_2=\{x+v_m+\omega \cdot (v_1, \ldots, v_{m-1})\}_{\bo \in \2^{m-1}}$. Any affine combination of forms in $T_2$ that gives $x$ must annihilate $v_m$ but then $x$ is annihilated with it. The first collection one can partition by the induction hypothesis.
\end{proof}

\begin{lemma}\label{Y_x}
For fixed $x$ the complexity of the system of affine forms in $\bar v$ corresponding to the points on the cube
\[
(x|v_1, v_2, \ldots, v_m)'
\]
is of complexity $\le m$ in $\bar v$.
\end{lemma}

\begin{proof} 
This is the same as the complexity of $(0|\bar v)'$ which is at most the complexity of $(0 |\bar v)$ which is $\le m$. 
 \end{proof}

\begin{lemma}\label{layer} Let $S=\{r_i(\bar x)\}$ be a collection of linear forms and let $S'=w+S=\{w+r_i(\bar x)\}$. 
If $S$ is of complexity $\le m$ in $\bar x$, then $S\cup S' \cup \{w\}$ is of complexity $\le m+1$ in $\bar x, w$. 
\end{lemma}

\begin{proof}   We consider three cases:\\
1) For the form $w$ take an admissible $m$ partition of $S' \setminus \{w\}$, add to it the set of all forms in $S$.\\
2) For a form $ r_i(\bar x)$ take an admissible $m$ partition of $S \setminus \{r_i(\bar x)\}$  and add $w+r_j(\bar x)$ for $j \ne i$ to the corresponding element of the partition that includes $r_j(\bar x)$; add $w$ to one of these sets, and add to the partition the singleton $\{w+r_i(x)\}$  to get an admissible $m+1$ partition. \\
3 ) Similarly for $w+r_i(\bar x)$, take an admissible $m$ partition of $S \setminus \{r_i(\bar x)\}$ and do the same. 
\end{proof}

 \begin{lemma}\label{double-u-v} Let $j \ge 1, m \ge 0$. The system  of linear  $v, v_1, \ldots, v_j, u, u_1, \ldots, u_j, w_1, \ldots, w_{m}$ corresponding to the points on the cubes
 \[
(v|v_1, \ldots, v_j, w_{1}, \ldots, w_m), (u|u_1, \ldots, u_j, w_{1}, \ldots, w_m)
 \]
is of complexity $\le j+m$.
 \end{lemma}
 
 \begin{proof}
 We prove this by induction on $m$. For $m=0$ the claim follows from the fact that $(v|v_1, \ldots, v_j),(u|u_1, \ldots, u_j)$ are of complexity $\le j$. 
 Now for $m>0$,  we assume the system 
 \[
(v|v_1, \ldots, v_j, w_{1}, \ldots, w_{m-1}), (u|u_1, \ldots, u_j, w_{1}, \ldots, w_{m-1})
 \]
 is of complexity $\le j+m-1$. The system corresponding to $m$ is  the system 
 \[\begin{aligned}
& (v|v_1, \ldots, v_j, w_{1}, \ldots, w_{m-1}), (u|u_1, \ldots, u_j, w_{1}, \ldots, w_{m-1})\\
 &(v+w_m|v_1, \ldots, v_j, w_{1}, \ldots, w_{m-1}),  (u+w_m|u_1, \ldots, u_j, w_{1}, \ldots, w_{m-1}),
\end{aligned} \]
 which by Lemma  \ref{layer} it is of complexity $\le (j+m-1)+1=j+m$. 
  \end{proof}

\begin{lemma}\label{add} Let $S=\{r_i(\bar x)\}$ be a collection of affine forms and let $S'=w+S=\{w+s_i(\bar x)\}$ be a system 
of complexity $\le m-1$ in $w, \bar x$. Then the complexity of the system $S \cup S'$ in the variables $w, \bar x$ at any form in $S'$  is $\le m$.
\end{lemma}

\begin{proof}  Let  $w+s_i(\bar x)$ be in $S'$. Take an admissible $m-1$ partition for $S'\setminus \{w+s_i(x)\}$ add to this partition the set $S$. Then this is an admissible
$m$ partition for $S \cup S' \setminus \{w+s_i(x)\}$. 
\end{proof}

\begin{lemma}\label{lin-comb} Let $v, w_i \in V$ and $A \in Aut(V)$ be such that $A(v)=v$. Then $v \in \text{span}(w_i)$ if and only if  $v \in \text{span}(Aw_i)$
\end{lemma}

\begin{proof} Assume $v= \sum a_i w_i$, then $Av=v=\sum a_i Aw_i$.
\end{proof}

For any map $\ep : [m]\to \pm 1, \bar v=(v_1,\ldots,v_m), \bo \in \2^m$ we write $\omega^\epsilon \cdot \bar v=\sum _{i=1}^m\epsilon (i)\omega _iv_i$.

\begin{lemma}\label{admisible}For any $\epsilon: [m] \to \pm 1$
\begin{enumerate}
\item The system of linear forms in   $ \bar y^0, \ldots,  \bar y^m \in G^m$
\[
(\omega^\epsilon  \cdot \bar y^0+ \nu  \cdot (\omega^\epsilon  \cdot \bar y^1,\ldots, \omega^\epsilon  \cdot \bar y^m))_{\bo    \in 2^m \setminus \bar 0 ,\  \nu \in 2^m}
 \]
is of complexity $\le m$.
\item 
The system in $\bar y^0 , \bar y^1, \ldots,  \bar y^m, \bar z^1, \ldots,  \bar z^m \in G^m$ corresponding to the cubes
\[\begin{aligned}
&( \omega^\epsilon  \cdot \bar y^0+ \nu  \cdot (\omega^\epsilon    \cdot \bar y^1,\ldots, \omega^\epsilon  \cdot \bar y^m))_{\bo  \in 2^m \setminus \bar 0,\  \nu \in 2^m}\\
&( \omega^\epsilon  \cdot \bar y^0+ \nu   \cdot (\omega^\epsilon  \cdot \bar z^1,\ldots, \omega^\epsilon  \cdot \bar z^m))_{\bo \in 2^m \setminus \bar 0,\  \nu \in 2^m}
\end{aligned} \]
is of complexity  $\le m$.
\item The system in the variables $y^0_1, \ldots, y^m_1$ ,  $y^0_2, \ldots, y^0_m,  y^1_2, \ldots, y^1_m, \ldots, y^m_2, \ldots, y^m_m$,  and  \\
$z^0_2, \ldots, z^0_m,  \ldots, z^m_2, \ldots, z^m_m$
\[\begin{aligned}
&(\omega^\epsilon  \cdot \bar y^0+ \nu  \cdot (\omega^\epsilon  \cdot \bar y^1,\ldots, \omega^\epsilon  \cdot \bar y^m))_{\bo \in 2^m \setminus \bar 0 ,\  \nu \in 2^m} \\
&(\omega^\epsilon  \cdot \bar z^0+ \nu  \cdot (\omega^\epsilon  \cdot \bar z^1,\ldots, \omega^\epsilon  \cdot \bar z^m))_{\bo \in 2^m \setminus \bar 0 ,\  \nu \in 2^m}
\end{aligned} \]
where $\bar z^i=(y^i_1, z^i_2, \ldots, z^i_m)$, 
 is of complexity  $\le m$.
\end{enumerate}
\end{lemma}

\begin{proof} 
We start with (1). We prove the claim by induction on $m$. For $m=1$ the claim is obvious. Since  CS-complexity is invariant under automorphisms, it suffices to prove this for $\epsilon=\bar 1$. Indeed, for  every $\epsilon \in \{-1,1\}^m $ there exists an invertible linear transformation on the variables which maps $\omega^{\epsilon}$ to $\omega$ for all $\omega \in \{0,1\}^m$. 

For $\bo_0=\bar 1$.  We write the collection of forms as a union of 
\[ \begin{aligned}
T=&\{(\omega1)^\ep \cdot \bar y^0+ (\nu0) \cdot ((\omega1)^\ep \cdot \bar y^1,\ldots, (\omega1)^\ep \cdot \bar y^m)\}_{\bo \in \2^{m-1},\  \nu \in \2^{m-1}, \bo \neq \tilde \bar 1} \\
S=&\{ \{(\omega1)^\ep \cdot \bar y^0+ (\nu1) \cdot ((\omega1)^\ep \cdot \bar y^1,\ldots, (\omega1)^\ep \cdot \bar y^m)\}_{\bo \in \2^{m-1} ,\  \nu \in \2^{m-1} } \\
&\ \{(\omega0)^\ep \cdot \bar y^0+ (\nu1) \cdot ((\omega0)^\ep \cdot \bar y^1,\ldots, (\omega0)^\ep \cdot \bar y^m)\}_{\bo \in \2^{m-1} \setminus \bar 0 ,\  \nu \in \2^{m-1}} \\
&\ \{(\omega0)^\ep \cdot \bar y^0+ (\nu0) \cdot ((\omega0)^\ep \cdot \bar y^1,\ldots, (\omega0)^\ep \cdot \bar y^m)\}_{\bo \in \2^{m-1}\setminus \bar 0 ,\  \nu \in \2^{m-1}}\}. \\
\end{aligned} \]
By the induction hypothesis $T$ has an admissible $m-1$ partition.   It remains to show that $\bar 1  \cdot \bar y^0$ is not in the linear span of the forms in $S$. 
 We rewrite $S$ as
 \[ \begin{aligned}
S=&\{ \{\ep(m)y^0_m +(\omega0) \cdot \bar y^0+ \ep(m)y^m_m+ (\omega0)^\ep \cdot \bar y^m+\nu \cdot ((\omega1)^\ep \cdot \bar y^1,\ldots, (\omega1)^\ep \cdot \bar y^{m-1})\}_{\bo \in \2^{m-1} ,\  \nu \in \2^{m-1} } \\
&\ \{(\omega0)^\ep \cdot \bar y^0+ (\nu1) \cdot ((\omega0)^\ep \cdot \bar y^1,\ldots, (\omega0)^\ep \cdot \bar y^m)\}_{\bo \in \2^{m-1} \setminus \bar 0 ,\  \nu \in \2^{m-1}} \\
&\ \{(\omega0)^\ep \cdot \bar y^0+ \nu \cdot ((\omega0)^\ep \cdot \bar y^1,\ldots, (\omega0)^\ep \cdot \bar y^{m-1})\}_{\bo \in \2^{m-1}\setminus \bar 0 ,\  \nu \in \2^{m-1}}\}. \\
\end{aligned} \]
Make the change of variable  $y^m_m \to y^m_m-y^0_m$. After this change of variable  $y^0_m$ does not appear in any one of the forms so 
$\bar 1  \cdot \bar y^0$ cannot be  in the linear span of the forms in $S$. \\

For (2)  we want to bound the complexity at any form in the new array.  Fix a vertex $(\bo_0, \nu_0)$ and take an admissible $m$ partition as in previous section for the original array. 
\[
( \omega^\ep \cdot \bar y^0+ \nu \cdot (\omega^\ep \cdot \bar y^1,\ldots, \omega^\ep \cdot \bar y^m))_{\bo \in \2^m \setminus \bar 0,\  \nu \in \2^m} \setminus \{\omega_0^\ep \cdot \bar y^0\}. 
\]
We wish to distribute the new forms into the sets in this partition. If $\bo_0=(1\bar 0)$ then we add $ \omega^\ep \cdot \bar y^0+ \nu \cdot (\omega^\ep \cdot \bar z^1,\ldots, \omega^\ep \cdot \bar z^m)$ to the partition element in which $ \omega^\ep \cdot \bar y^0+ \nu \cdot (\omega^\ep \cdot \bar y^1,\ldots, \omega^\ep \cdot \bar y^m)$ resides.  Otherwise, for any $(\bo, \nu) \neq (\bo_0, \nu_0)$ we add $ \omega^\ep \cdot \bar y^0+ \nu \cdot (\omega^\ep \cdot \bar z^1,\ldots, \omega^\ep \cdot \bar z^m)$ to the partition element in which $ \omega^\ep \cdot \bar y^0+ \nu \cdot (\omega^\ep \cdot \bar y^1,\ldots, \omega^\ep \cdot \bar y^m)$ resides.  We make the observation that in the partition described in (1) there is a partition element that contains $3$ linear forms that together with the chose form corresponding to the vertex $(\bo_0, \nu_0)$  correspond to $4$ independent variable  (in the case $m=2$ described in detail this would be the linear forms in $T$.)
We add
 $ \omega^\ep_0 \cdot \bar y^0+ \nu_0 \cdot (\omega^\ep \cdot \bar z^1,\ldots, \omega^\ep \cdot \bar z^m)$  to this partition element,

For (3) the argument is similar to $2$.
\end{proof}

%%%%%%%%%%%%%%%%%%%%%%%%%%%%%%%%%%%%%%%%%%%%%%%%%
\begin{section}{Counting Lemmas}

We can count the number of various systems of affine configurations in uniform sets.
\begin{lemma}\label{counting} Let $s>0$. 
Let $\bar s=\{s_i(\bar x, \bar y)\}_{i=1}^r$ on $G^m\times G^n$  be a non degenerate system of affine linear forms, and let $X \subset G$ be $(\eta,m)$-uniform of density $\delta$
\begin{enumerate}
\item If $\bar s $ is of complexity $\le m$ then
$$|\{ \bar x, \bar y: s_i(\bar x, \bar y) \in X\}| = (\delta^{r}+O(\eta))|G|^{n+m}.$$

\item If $\{s_i(\bar x, \bar y), s_i(\bar x, \bar y')\}$ is of complexity $\le m$ then for 
$O(\sqrt \eta)$ a.e $\bar x \in G^m$ we have 
$$|\{\bar y: s_i(\bar x, \bar y) \in X\}| =  (\delta^{r}+O(\eta^{1/4}))|G|^{n}.$$

\item If $\{r_j(\bar x)\}_{j \in [t]} \cup \{s_i(\bar x, \bar y), s_i(\bar x, \bar y')\}_{i \in [r]}$ is of complexity $\le m$ then for 
$O(\delta^{-t}\sqrt \eta)$ a.e $\bar x$ such that $\{r_j(\bar x)\} \in X$ we have 
$$|\{\bar y: s_i(\bar x, \bar y) \in X\}| =  (\delta^{r}+O(\eta^{1/4}))|G|^{n}.$$
\end{enumerate}
\end{lemma}

\begin{proof} 
(1) We estimate first the size of $\{\bar x, \bar y: s_i(\bar x, \bar y) \in X\}$. The number of points is given by
\[
 \sum_{\bar x, \bar y} \prod_{i=1}^r1_X(s_i(\bar x, \bar y)). 
\]
We estimate 
$
| \mE_{\bar x, \bar y}\prod_i 1_X(s_i(\bar x, \bar y)) - \delta^r|.$
Note that we can write 
\[
 \mE_{\bar x, \bar y}\prod_i 1_X(s_i(\bar x, \bar y))
 =  \mE_{\bar x, \bar y} (1_X(s_1(\bar x, \bar y))-\delta) \prod_{i>1} 1_X(s_i(\bar x, \bar y))
 + \delta \mE_{\bar x, \bar y} \prod_{i>1} 1_X(s_i(\bar x, \bar y))
\]
Repeating this we can write $\mE_{\bar x, \bar y}\prod_i 1_X(s_i(\bar x, \bar y)) -\delta^r$ as a sum of $r$ terms 
 form
\[
 \mE_{\bar x, \bar y} \prod_{i=1}^rg_i(s_i(\bar x, \bar y)),
\]
where in each term we have  $g_i=1_X-\delta$ for at least one $i$ and is thus bounded by $O(\|1_X-\delta \|_{U_m})$.  \\
\ \\
(2) Consider the average
\[
 \mE_{\bar x}\big|\mE_{ \bar y} \prod_{i=1}^r1_X(s_i(\bar x, \bar y)) - \delta^r\big|
\]
The inner average can be written as a a sum of  $r$ terms of the form
$
 \mE_{\bar y} \prod_{i=1}^rg_i(s_i(\bar x, \bar y)),
$
where in each term we have  $g_i=1_X-\delta$ for at least one $i$, 
and we can bound
\[
\mE_{\bar x}|  \mE_{\bar y} \prod_{i=1}^rg_i(s_i(\bar x, \bar y))|^2  \le \mE_{\bar x,\bar y, \bar y'} \prod_{i=1}^rg_i(s_i(\bar x, \bar y))\bar g_i(s_i(\bar x, \bar y'))  =O(\|1_X-\delta \|_{U_m}). 
\]
It follows that for $O(\sqrt \eta)$ a.e  $\bar x$ we have 
\[
\big|\mE_{ \bar y} \prod_{i=1}^r1_X(s_i(\bar x, \bar y)) - \delta^r\big|^2=O( \sqrt \eta).
\]
\\
(3) Consider the average
\[
 \mE_{\bar x} \prod_j1_X(r_j(\bar x))\big|\mE_{ \bar y} \prod_{i=1}^r1_X(s_i(\bar x, \bar y)) - \delta^r\big|
\]
and proceed as in $(2)$ to obtain that 
for $O(\delta^{-t}\sqrt \eta)$ a.e  $\bar x$ such that $\prod_j1_X(r_j(\bar x)) =1$ we have 
\[
\big|\mE_{ \bar y} \prod_{i=1}^r1_X(s_i(\bar x, \bar y)) - \delta^r\big|^2=O( \sqrt \eta).
\]

\end{proof}

\begin{definition}[$Y_x$]
We denote by $Y_x$ the set
\[Y_x = \{\bar v: (x| \bar v)'\in C_m(X)'\} .  
\]
\end{definition}

\begin{lemma}\label{size} If $X$ is $(\eta, m)$-uniform then  $|Y_x|=(\delta^{2^{m}-1}+O(\eta))|V|^m$.
\end{lemma}

\begin{proof}
The system of affine linear forms corresponding  to $Y_x$ is of complexity  $\le m$  by Lemma \ref{Y_x}. 
\end{proof}

 Below are some  lemmas that are finitary analogues of measure theoretic properties. Let $p:Z\to Y$ be a map between finite sets. 
 For $y \in Y$ denote by $S_y:=p^{-1}(y)$.  We say that a map $p$  is {\it $C$-homogeneous}, $C\geq 1$ 
if  $|S_y|/|S_{y'}|\leq C, \forall y,y'\in Y$ where $S_y:=p^{-1}(y)$.
Let $p:Z\to Y$ be a {\it $C$-homogeneous} map and $\epsilon $ be a positive number.

Let $P\subset Z$ be a subset 
such that $|P|/|Z|\geq 1-\epsilon$. Let $M>0$. We define $Q\subset Y$ by 
$$Q=\{y\in Y: |S_y\cap P|/|S_y|\geq 1-C^2 M\epsilon\}$$ 

\begin{lemma}[Fubini]\label{fubini}  $|Q|/|Y|\geq 1-1/M$.
\end{lemma}

\begin{proof}
Let $A=P^c$, and let $A_{y}=  S_{y} \cap P^c$.  Let $B=Q^c$, and let $S=|S_{y_0}|$ for some $y_0 \in Y$. Then
\[
\epsilon CS|Y|\ge  \epsilon|Z| \ge |A| \ge  \sum_{y \in B} |A_y|   \ge  \sum_{y \in B} |S_y| C^2 M\epsilon \ge  |B|SCM\epsilon.
\]
\end{proof}

In particular we also have  $$|\{y\in Y: |S_y\cap P|/|S_y|\geq 1-C^2 \sqrt \epsilon\}|/|Y|\geq 1-\sqrt \ep,$$
and 
\[
|\{y\in Y: |S_y\cap P|/|S_y|\geq 1-C^2/M\}|/|Y|\geq 1-M\epsilon.
\]

The next Lemmas are  immediate: 

\begin{lemma}\label{density-ae} Let $B \subset A$ with $|B| \ge c|A|$ suppose a property $P$ holds $\epsilon$ a.e. $x \in A$ then 
$P$ holds for $\epsilon/c$-a.e. $x \in B$.  
\end{lemma}

\begin{lemma}\label{ae-projection1} Let $\epsilon, c>0$. Let  $p:X\to Y$, and suppose $Y' \subset Y$ with $|Y'| \le \epsilon |Y|$, and for all $y \in Y$ we have $cS \le S_y \le S$. Then $|p^{-1}(Y')| \le \frac{\epsilon}{c}|X|$.
\end{lemma}

\begin{lemma}\label{ae-projection} Let $0 \le \epsilon, \eta, \delta <1/2$. Let  $p:X\to Y$. Suppose $Y' \subset Y$,  $|Y'| \le \epsilon |Y|$, and $|S_y -S| \le \delta S$ for all $y \in Y''$ with $Y'' \subset Y$ of size 
$|Y''|> (1-\eta)|Y|$, and $|S_y| \le CS$ for all $y \in Y$. Then  $|p^{-1}(Y')| \le  |X|(8\epsilon+4C\eta)$. 
\end{lemma}

\begin{proof} On the one hand
\[
|p^{-1}(Y')| = \sum_{y \in Y' \cap Y''} |S_y| +  \sum_{y \in Y' \cap (Y'')^c} |S_y| \le \epsilon|Y|(1+\delta)S + \eta|Y|CS = |Y|S(\epsilon(1+\delta)+C\eta).
\]
on the other hand
\[
|X| = \sum_{y \in Y} |S_y| =\sum_{y \in  Y''} |S_y| +  \sum_{y \in (Y'')^c} |S_y|  \ge (1-\eta)|Y|(1-\delta)S.
\]
Together we get
\[
|p^{-1}(Y')| \le  |X|(\epsilon(1+\delta)+C\eta)/(1-\eta)(1-\delta).
\]
\end{proof}

Lemmas \ref{ae-projection1}, \ref{ae-projection} allows us to pull back good properties of an image of a map to the source.

\end{section}

%%%%%%%%%%%%%%%%%%%%%%%%%%%%%%%%%%%%%%%%%
\section{Proof of Theorem \ref{testing-uniform}.}

Let $m \ge 1$, and let $\delta,\eta,  \epsilon >0$. Let  $X$ be an $(\eta, m)$-uniform subset of $G$ of density $\delta$  and let $f:X \to H$ be with $f_m(c)=0$ for $\epsilon$-a.e.  $c \in C_m(X)$.  All $O(1), \Omega(1)$ in this section are constants that depend {\em only} on $m$, and we suppress this dependence. 
Fix $a \in V$. Recall that $Y_a=\{\bar v: (a|\bar v) \in C_m(X)\}$, and that uniformity of $X$ ensures that $Y_a$ is a large set, 
For  $\bar v \in Y_a$ denote
\[
F_a(\bar v) = \sum_{\omega \in \{0,1\}^{m}\setminus \bar 0}(-1)^{|\omega|} f(a+\omega \cdot \bar v) 
\]

The main step is to show that for $\eta$ sufficiently small $F_a$ is constant  for $O( \epsilon)$ a.e. $\bar v \in Y_a$ (Proposition \ref{constant}). 
This allows one to define $h(a)$ as the common value of $F_a(\bar v)$. 
To show this we compare the value of $F_a$ at two different point  $\bar u, \bar v$ in $Y_a$ and show that the difference vanishes almost surely. We write the difference in many ways as an alternating sum of $f_m$ evaluated at a bounded collection of cubes. 
To be able to make use of the fact that $f_m$ vanishes almost surely on $X$ we need control over the parameters involved in the different ways of writing the difference as a sum of cubes - this is the main difficulty.  Next we define $h(a)$ as the common value of $F_a(\bar v)$ and show that $h_m$ vanishes on $C_m(V)$ (Proposition \ref{hmzero}). Finally we show that $h=f$ almost surely.  

We explain the strategy in  more detail in the case when $m=2$. In this case for any $w_1, w_2$  we can write
\begin{equation}\label{ex1}\begin{aligned}
F_a(v_1, v_2 )-F_a(u_1,u_2 )=& f_2(a+u_1|w_1-u_1, u_2) +  f_2(a+u_2|w_1, w_2- u_2) \\
&-  f_2(a+v_1|w_1-v_1, v_2) +  f_2(a+v_2|w_1, w_2- v_2).
\end{aligned}\end{equation}
But we can not choose $w_1, w_2$ in an arbitrary way since we are given that $f_2$ vanishes almost surely only for cubes in $X$.
To be able to use this we look at the four maps  $p^u_1, p^u_2, p_1^v, p_2^v: Y_a^2 \times V^2 \to C_2(V)$ taking $(v_1, v_2, u_1,u_2, w_1, w_2)$ to
\[
(a+u_1|w_1-u_1, u_2), (a+u_2|w_1, w_2- u_2), (a+v_1|w_1-v_1, v_2), (a+v_2|w_1, w_2- v_2)
\] 
respectively and show that for almost any $2$-cube in $C_2(X)$ the fibers of each maps are of essentially the same size. 
Denote by $A$ the set of $(v_1, v_2, u_1,u_2, w_1, w_2)$  in  $Y_a^2 \times V^2$ such that the image of all four maps is in $C_2(X)$. Uniformity implies that this set is large.   Since $f_2$ vanishes on almost surely on $C_2(X)$ we can deduce that for almost any $(v_1, v_2, u_1,u_2, w_1, w_2)$  in  $A$ the expression on the right hand side of equation \eqref{ex1} vanishes. It then remains to  show that for almost any $(v_1, v_2, u_1,u_2) \in Y_a^2$ there are many $w_1, w_2$ such that $(v_1, v_2, u_1,u_2, w_1, w_2) \in A$. 

The next step is then to show that $h_2$ vanishes on all $2$-cubes . We write $h_2(a|a_1, a_2)$ in many ways as an alternating sum of $h_2$ evaluated of fours cubes 
\[(a|y^0_1,y^0_2), (a+a_1|y^0_1+y^1_1,y^0_2+y^1_2),(a+a_2|y^0_1+y^2_1,y^0_2+y^2_2),(a+a_1+a_2|y^0_1+y^1_1+y^2_1,y^0_2+y^1_2+y^2_2)
\]
Once again we can no choose $y^0_1, y^0_2,y^1_1,y^1_2,y^2_1,y^2_2$ in an arbitrary way - we need to choose them such that 
all the points in all cubes fall in $X$ and also such that the cubes fall into the set where $f_m$ vanishes and where $h(a) =F_a(y^0_1, y^0_2)$ (and similarly for $h(a+a_1), h(a+a_2),h(a+a_1+a_2)$). Once again we use uniformity of $X$ to get control over the fibers of the maps associated with these conditions. Finally we show that $h=f$ almost surely. \\

We turn to the details of the proof. We start by showing that  for $\eta$ sufficiently small $F_a$ is constant  a.e. $\bar v \in Y_a$.
 
\begin{proposition}\label{constant}  There exists $B>0$ depending on $m$ such that for $\eta<(\epsilon\delta)^{B}$ the function  $F_a(\bar v )$ is constant $O( \epsilon)$ a.e. $\bar v \in Y_a$.
\end{proposition}

\begin{proof}
Fix $a \in V$. Observe that for any $w \in V$ we have 
\[
 f_m(a|\bar v) = f_m(a| w, v_2, \ldots, v_m)- f_m (a+v_1| w-v_1, v_2, \ldots, v_m),
\]
and similarly for any $i=2, \ldots, m$. 
Note that if $(a| w, v_2, \ldots, v_m),  (a+w| v_1-w, v_2, \ldots, v_m) \in C_m(X)$ then so is $ (a|\bar v)$.  

Thus we get that  $w_1, \ldots, w_m \in G$ we have 
\[\begin{aligned}
&f_m(a|v_1, \ldots, v_m)= f_m(a|w_1,v_2, \ldots, v_m)-f_m(a+v_1|w_1-v_1,v_2, \ldots, v_m)  \\
&= f_m(a|w_1,w_2, \ldots, v_m)- f_m(a+v_2|w_1,w_2-v_2, \ldots, v_m)-f_m(a+v_1|w_1-v_1,v_2, \ldots, v_m) \\
&=\qquad  \vdots \\
&= f_m(a|w_1,w_2, \ldots, w_m) - \sum_{i=1}^m  f_m(a+v_i|w_1,\ldots, w_{i-1}, w_i-v_i,v_{i+1} \ldots, v_m),
\end{aligned}\]
so that
\[\begin{aligned}
&F_a(\bar v )-F_a(\bar u )= \\
& \sum_{i=1}^m  f_m(a+u_i|w_1,\ldots, w_{i-1}, w_i-u_i,u_{i+1} \ldots, u_m)
- \sum_{i=1}^m f_m (a+v_i|w_1,\ldots, w_{i-1}, w_i-v_i,v_{i+1} \ldots, v_m)
\end{aligned}
\]
Consider the collection of affine forms associated with the cubes
\[
(*) \quad \{(a+u_i|w_1,\ldots, w_{i-1}, w_i-u_i,u_{i+1} \ldots, u_m), (a+v_i|w_1,\ldots, w_{i-1}, w_i-v_i,v_{i+1} \ldots, v_m)\}_{i=1}^m  
\]

\begin{lemma}
1) The system $(*) \cup \{a\}$ is of complexity $\le m$ in $a, \bar v, \bar u, \bar w$. \\
2) For fixed $a$, the system $(*)$ is of complexity $\le m$ in $\bar v, \bar u, \bar w$.
\end{lemma}

\begin{proof}
Note that  the system in 1) is equivalent to the collection of  forms associated with the cubes $(a| w_1,\ldots ,w_{i-1},u_i,\ldots ,u_m)$ and $(a| w_1,\ldots, w_{i-1},v_i, \ldots ,v_m)$ for $i=1, \ldots, m$, and the system in $2)$ to 
$(0| w_1,\ldots ,w_{i-1},u_i,\ldots ,u_m)'$ and $(0| w_1,\ldots, w_{i-1},v_i, \ldots ,v_m)'$ for $i=1, \ldots, m$. 
The following claim would conclude the proof:

\begin{claim} Let $\{z_1,\ldots,z_t\}$ be variables. The system of linear forms $\{\sum_{i \in T} z_i: |T|\le m\}$ has CS complexity $\le m$. Similarly, if $z_0$ is a new variable then $\{z_0+\sum_{i \in T} z_i: T \subset \{1, \ldots, t\}, |T|\le m\}$ has CS complexity $\le m$.
 \end{claim}

\begin{proof} We prove this by induction on $m$ and $t$. For $m=1$ any $t$ the claim is obvious. For any $m$, $t \le m$ the claim follows from  Lemmas \ref{cube-complexity}, \ref{Y_x}. Fix $m>1$ and assume the claim for $t \ge m$.  Consider now the system
on $t+1$ variables $\{z_1,\ldots,z_{t+1}\}$, $\{\sum_{i \in T} z_i: |T|\le m\}$. We can write this as $S_1 \cup S_2$ where 
\[
S_1=\{\sum_{i \in T} z_i: T \subset  [t]\}, \ S_2= \{z_{t+1}+ \sum_{i \in T} z_i: T \subset [t], |T| \le m-1\}
\]
By symmetry it suffices to show that the complexity of any form in $S_2$ is $\le m$. By the induction hypothesis the collection $S_2$ 
 is of complexity $\le m-1$. Fix  a form in $S_2$ and take an $m-1$ partition for the rest of the forms in $S_2$. Now add to this the set $S_1$ to obtain an $m$-partition. 
\end{proof}
Let $Z= \{u_1,\ldots,u_m, v_1,\ldots,v_m,w_1,\ldots,w_m\}$. We apply the Claim to $Z \cup \{a\}$ and $Z$ respectively. 
\end{proof}

By Lemma \ref{counting} we get:
\begin{corollary}
 The set 
\[ \begin{aligned}
A=\{(\bar u, \bar v, \bar w): &\{(a+u_i|w_1,\ldots, w_{i-1}, w_i-u_i,u_{i+1} \ldots, u_m), \\
&(a+v_i|w_1,\ldots, w_{i-1}, w_i-v_i,v_{i+1} \ldots, v_m)\}_{i=1}^m 
 \in C_m(X)\}
\end{aligned}\]
is of size $(\delta^{O(1)} +O(\eta))|V|^{3m}$.
\end{corollary}

\begin{lemma}\label{A} For $(O(\epsilon)+O(\delta^{-O(1)}\eta^{\Omega(1)}))$-a.e $(\bar u, \bar v, \bar w) \in A$ we have 
\[
(*) \quad f_m(a+u_i|w_1,\ldots, w_{i-1}, w_i-u_i,u_{i+1} \ldots, u_m)=f_m(a+v_i|w_1,\ldots, w_{i-1}, w_i-v_i,u_{i+1} \ldots, v_m)=0
\]
for all $1 \le i \le m$. 
\end{lemma}

\begin{remark} If one is interested in a weaker form of Theorem \ref{testing-uniform} where $\alpha$ is allowed to depend on $\delta$ (polynomially)
then  one can obtain the Lemma quickly by looking at the maps  $p^u_i, p^v_i: V^{3m} \to  C_m(V)$
\[
p^u_i: (\bar u, \bar v, \bar w) \mapsto (a+u_i|w_1,\ldots, w_{i-1}, w_i-u_i,u_{i+1} \ldots, u_m).
\]
and similarly $p^v_i$. For each such map the sieve over a point in $C_m(X)$ is of size $|V|^{2m-1}$. Since $|A|= (\delta^{O(1)} +O(\eta))|V|^{3m}$ we get 
 that  $O(\delta^{-O(1)}\epsilon)$ a.e   $(\bar u, \bar v, \bar w) \in A$ we have $f_m(a+u_i|w_1,\ldots, w_{i-1}, w_i-u_i,u_{i+1} \ldots, u_m)$. By the union bound this holds for all $i$. 
\end{remark}

\begin{proof}[Proof of Lemma \ref{A}]
We prove this by induction. The first step is the following Lemma:
\begin{lemma}\label{A_m}
Consider the set
\[
A_m=\{ u_m, v_m, w_1, \ldots, w_m  : (a+u_m|w_1,\ldots, ,w_{m-1}, w_m-u_m), (a+v_m|w_1,\ldots, w_{m-1}, w_m-v_m) \in C_m(X)\}
\]
Then for $(O(\epsilon)+O(\delta^{-O(1)}\eta^{\Omega(1)}))$ a.e. $u_m, v_m , \bar w \in A_m$ we have
$$f_m(a+u_m|w_1,\ldots, w_{m}-u_m) =f_m(a+v_m|w_1,\ldots, w_{m}-v_m) =0.$$ 
\end{lemma}

\begin{proof}
Consider the map $p_m:V^{m+2} \to C_m(V)$ defined by  
\[
(u_m, v_m , \bar w) \mapsto (a+u_m|w_1,\ldots,w_{m-1}, w_m-u_m).
\] 
Fix a cube 
$(s|\bar t) \in C_m(X)$ and consider the intersection $p_m^{-1}((s|\bar t)) \cap A_m$.  We will show that  the sizes of these fiber are almost surely of essentially the same size.  More precisely we show that for $O(\delta^{-O(1)} \sqrt \eta)$ a.e. $(s|\bar t) \in C_m(X)$  we have $p_m^{-1}((s|t)) \cap A_m$ is of size $((\delta^{2^{m-1}})+O(\eta^{1/4}))|V|$.  It will then follow by Lemma \ref{ae-projection} that  for $(O(\epsilon)+O(\delta^{-O(1)}\eta^{1/4}))$ a.e. $u_m, v_m , \bar w \in A_m$ we have
$f_m(a+u_m|w_1,\ldots, ,w_{m-1}, w_m-u_m) =0$. Similarly we will have that  for $(O(\epsilon)+O(\delta^{-O(1)}\eta^{1/4}))$ a.e. $u_m, v_m , \bar w \in A_m$,  have
$f_m(a+v_m|w_1,\ldots, ,w_{m-1}, w_m-v_m) =0$. Thus for $(O(\epsilon)+O(\delta^{-O(1)}\eta^{1/4}))$ a.e. $u_m, v_m , \bar w \in A_m$ we have
$$f_m(a+u_m|w_1,\ldots, ,w_{m-1}, w_m-u_m) =f_m(a+v_m|w_1,\ldots, ,w_{m-1}, w_m-v_m) =0.$$ 
\ \\
We now show that the sizes of the fibers  $p_m^{-1}((s|\bar t)) \cap A_m$ are almost surely of essentially the same size. 
Consider the system of affine forms
\[
 (a+u_m|w_1,\ldots, ,w_{m-1}, w_m-u_m), (a+v_m|w_1,\ldots, w_{m-1}, w_m-v_m) ,  (a+v'_m|w_1,\ldots, w_{m-1}, w_m-v'_m). 
\]
We claim that this system is of complexity $m$. Indeed  we can rewrite this system as 
\[
 (a+u_m|w_1,\ldots, ,w_{m-1}, w_m-u_m), (a+v_m|w_1,\ldots, w_{m-1}) ,  (a+v'_m|w_1,\ldots, w_{m-1}). 
\]
Since $ (a+u_m|w_1,\ldots, ,w_{m-1}, w_m-u_m)$ is of complexity $m$, and  $(a+v_m|w_1,\ldots, ,w_{m-1})$ is of complexity $m-1$, and both together are a subset of the forms in $A$ so of complexity $\le m$. by Lemma \ref{double} the above system is of complexity $m$. It follows by Proposition \ref{counting} that for $O(\delta^{-O(1)}\sqrt \eta)$ a.e. $(s|\bar t) \in C_m(X)$  we have $p_m^{-1}((s|t)) \cap A_m$ is of size $(\delta^{2^{m-1}}+O(\delta^{-O(1)}\eta^{1/4}))|V|$. 
\end{proof}

Note that $A_m$ is a parametrization for $B_m$ - the set of all pairs of $m$ cubes $c, c'$ in $X$ that share an $m-1$ dimensional face, and by the above lemma $f_m(c)=f_m(c')=0$ on $(O(\epsilon)+O(\delta^{-O(1)}\eta^{\Omega(1)}))$ a.e. on such configurations.

We proceed by induction: Let 
\[ \begin{aligned}
A_j=\{(u_j, \ldots, u_m, v_j, \dots, v_m,  \bar w): &\{(a+u_i|w_1,\ldots, w_{i-1}, w_i-u_i,u_{i+1} \ldots, u_m), \\
&(a+v_i|w_1,\ldots, w_{i-1}, w_i-v_i,v_{i+1} \ldots, v_m)\}_{i\ge j} 
 \in C_m(X)\}
\end{aligned}\]

Consider $B_j=B_j(X)$ - collection of configurations  of $2(m-j+1)$ $m$-cubes $(c_i)_{1=1}^{2(m+j-1)}$ in $C_m(X)^{2(m+j-1)}$ associated with  
with the forms
\[\{(a+u_i|w_1,\ldots, w_{i-1}, w_i-u_i,u_{i+1} \ldots, u_m), 
(a+v_i|w_1,\ldots, w_{i-1}, w_i-v_i,v_{i+1} \ldots, v_m)\}_{i\ge j},
\]
and let $B_j(V)$ be the same collection of configurations with points in $V$ (this is just the solutions to a collection of linear equations on  $V$).

We assume that for $(O(\epsilon)+O(\delta^{-O(1)}\eta^{\Omega(1)}))$-a.e  $(c_i)_{1=1}^{2(m+j-1)} \in B_j$ , we have  $f_m(c_i)=0$ for all $i$. 

Now we wish to show that for $(O(\epsilon)+O(\delta^{-O(1)}\eta^{\Omega(1)}))$ a.e. $(u_{j-1}, \ldots, u_m, v_{j-1}, \dots, v_m,  \bar w)\in A_{j-1}$ we have
$$f_m(a+u_i|w_1,\ldots, w_{i-1}, w_i-u_i,u_{i+1} \ldots, u_m)= f_m(a+u_i|w_1,\ldots, w_{i-1}, w_i-u_i,u_{i+1} \ldots, u_m)=0$$
for $i\ge j-1$. 

We observe that the forms in $A_{j-1} \cup \{(a|w_1, \ldots, w_m)'\}$ is a subset of the forms in $A$ and thus of complexity $m$. Furthermore, the forms in $A_{j-1}$ are the union of the forms in $A_j$ and the forms  associated with the cubes
\[
(a+u_{j-1}|w_1,\ldots,w_{j-2}, u_j , \ldots, u_{m}), (a+v_{j-1}|w_1,\ldots,w_{j-2}, v_j ,\ldots, v_{m})
\]

In the case when $m=2$,
\[
A_2=\{(u_2, v_2,w_1, w_2): (a+u_2|w_1,w_2-u_2), (a+v_2|w_1,v_2-u_2) \in C_m(X)\}
\]
and we adjoin to the system of  forms in $A_2 \cup \{(a|w_1, w_2)'\}$ the forms  associated to the cubes
\[
 (a+u_1|w_1-u_1,u_2),  (a+v_1|w_1-v_1,v_2) 
\]
So the new forms we add are  the forms 
\[
a+u_2+w_1, a+u_1+u_2, a+v_2+w_1, a+v_1+v_2, 
\]
which is a system of complexity $1$ in $w_1, u_1, u_2, v_1, v_2$.

Now for $m\ge 2$, $A_{j-1} \cup \{(a|w_1,\ldots, w_m)'\}$ is obtain from $A_j \cup \{(a|w_1,\ldots, w_m)'\}$ by adding the forms 
\[
(a+u_{j-1}|w_1,\ldots,w_{j-2}, u_j , \ldots, u_{m}), (a+v_{j-1}|w_1,\ldots,w_{j-2}, v_j ,\ldots, v_{m}).
\]
and the new forms we add are which is a system of complexity $m-1$ by Lemma \ref{double-u-v}.

Consider now the system of forms in $A_j \cup \{(a|w_1,\ldots, w_m)'\}$  adjoined with the forms associated to the cubes
\begin{equation}\label{ooof}\begin{aligned}
&(a+u_{j-1}|w_1,\ldots,w_{j-2}, u_j , \ldots, u_{m}), (a+v_{j-1}|w_1,\ldots,w_{j-2}, v_j ,\ldots, v_{m}) \\
&(a+u'_{j-1}|w_1,\ldots,w_{j-2}, u_j , \ldots, u_{m}), (a+v'_{j-1}|w_1,\ldots,w_{j-2}, v_j ,\ldots, v_{m})
\end{aligned}\end{equation}
By Lemma \ref{double} this system is of complexity $\le m$, and thus $A_{j-1}$  adjoined with the forms in \eqref{ooof}
is of complexity $ \le m$. 

Consider the natural map $p_{j-1}:V^{m+2(m-j+2)} \to B_j(V)$. Fix a configuration of cubes $c$  in 
$ B_j(X)$ and consider the intersection $p_m^{-1}(c) \cap A_{j-1}$.   It follows by Proposition \ref{counting} that for $O( \delta^{-O(1)}\sqrt \eta)$ a.e. $c \in B_j(X)$  we have $p_{j-1}^{-1}(c) \cap A_{j-1}$ is of size $((\delta^{2\cdot 2^{m-1}})+O(\eta^{1/4}))|V|^2$. Thus by Lemma \ref{ae-projection} for $(O(\epsilon)+O(\delta^{-O(1)}\eta^{\Omega(1)}))$ a.e.  $(u_{j-1}, \ldots, u_m, v_{j-1}, \dots, v_m,  \bar w)\in A_{j-1}$ we have 
\[
f_m(a+u_i|w_1,\ldots, w_{i-1}, w_i-u_i,u_{i+1} \ldots, u_m)=f_m(a+v_i|w_1,\ldots, w_{i-1}, w_i-v_i,u_{i+1} \ldots, v_m)=0
\]
for all $j-1 \le i \le m$. 
\end{proof}

Let $A'$ denote the set of $(\bar u, \bar v, \bar w) \in A$ for which $(*)$ in Lemma \ref{A} holds. Then $|A'|=(O(\epsilon)+O(\delta^{-O(1)}\eta))|A|$. 
For $\bar u, \bar v \in Y_a$, consider the sets:
\[
A_{\bar v, \bar u} =\{ \bar w: (\bar u, \bar v, \bar w) \in A\}.
\]
It remains to show that for $(O(\epsilon)+O(\delta^{-O(1)}\eta^{\Omega(1)}))$ a.e. $\bar u,\bar v \in Y_a$ the set $A_{\bar v, \bar u}  \cap A'$ is not empty. 
 We prove  by induction that $O(\delta^{-O(1)}\eta^{\Omega(1)}))$ a.e. $\bar u, \bar v \in Y_a$, $|A_{\bar v, \bar u}|= (\delta^C+O(\eta))|V|^m$, for some $C>0$;
 in particular not empty for $\eta$ sufficiently small (polynomially in $\delta$).

\begin{lemma}
For $O(\delta^{-O(1)}\eta^{\Omega(1)}))$ a.e. $\bar u,\bar v \in Y_a$ 
the set  
\[
D_1(\bar u, \bar v)= \{ w_1 \in V: (a+u_1|w_1-u_1,\ldots, ,u_m),(a+v_1|w_1-v_1,\ldots ,v_m) \in C_m(X)\}
\]
 is of size $(\delta^C+O(\eta))|V|$.
\end{lemma}

\begin{proof}
Define the set
\[
D_1= \{ w_1 \in V, \bar u, \bar v \in Y_a: (a+u_1|w_1-u_1,\ldots, ,u_m),(a+v_1|w_1-v_1,\ldots ,v_m) \in C_m(X)\}
\]
Consider the map $r: D_1 \to  Y_a^2$ defined by $(w_1, \bar u, \bar v) \mapsto  (\bar u, \bar v)$.  The collection of forms in $D_1$ is of complexity $m$ as a subset of the forms in $A$. Further more, we can write the forms in $D_1$ as the union of the forms associated with $(a| \bar v)', (a|\bar u)'$ and the firms associated with 
\begin{equation}\label{qqq}
(a+w_1| u_2, \ldots, u_m), (a+w_1| v_2, \ldots, v_m).
\end{equation}
Note that the forms in \eqref{qqq} are a system of complexity $m-1$, thus by Lemma \ref{double}, the system 
\[
(a| \bar v)', (a|\bar u)', (a+w_1| u_2, \ldots, u_m), (a+w_1| v_2, \ldots, v_m).
(a+w'_1| u_2, \ldots, u_m), (a+w'_1| v_2, \ldots, v_m).
\]
is of complexity $\le m$ and by Lemma \ref{counting} $O(\delta^{-O(1)}\sqrt \eta))$ a.e. $\bar u,\bar v \in Y_a$ 
the set  
\[
D_1(\bar u, \bar v)= \{ w_1 \in G: (a+u_1|w_1-u_1,\ldots, ,u_m),(a+v_1|w_1-v_1,\ldots ,v_m) \in C_m(X)\}
\]
is of size $(\delta^C+O(\eta^{1/4}))|G|$.
\end{proof}

Now consider the set.
\[\begin{aligned}
D_j= \{ w_1, \ldots, w_j \in G, \bar u, \bar v \in Y_a: &(a+u_i|w_1,\ldots,w_{i-1}, w_i-u_i ,u_{i+1}, \ldots, u_m), \\
&(a+v_i|w_1,\ldots,w_{i-1}, w_i-v_i ,v_{i+1}, \ldots, u_m)\in C_m(X), 1 \le i\le j\}
\end{aligned}\]

For $(\bar u, \bar v, w_1, \ldots, w_{j-1}) \in D_{j-1}$ consider the set
\[\begin{aligned}
D_j(\bar u, \bar v, w_1, \ldots, w_{j-1}) = \{ w_{j} \in V: &(a+u_j|w_1,\ldots,w_{j-1}, w_j-u_j ,u_{j+1}, \ldots, u_m),\\
&(a+v_j|w_1,\ldots,w_{j-1}, w_j-v_j ,v_{j+1}, \ldots, u_m) \in C_m(X)\}
\end{aligned}
\]
by the same argument as above we get that $(1+O(\delta^{-O(1)}\sqrt \eta))$ a.e. $(\bar u, \bar v, w_1, \ldots, w_{j-1}) \in D_{j-1}$
$D_j(\bar u, \bar v, w_1, \ldots, w_{j-1})$ is of size $(\delta^C+O(\eta^{1/4}))|V|$. This completes the proof of Proposition \ref{constant}.
\end{proof}

\begin{definition}
Define $h(a)$ to be the common values of $F_a(\bar v)$. 
\end{definition}

\begin{proposition}\label{hmzero} There exist $\alpha, R>0$ such that for any $\epsilon< \alpha$  if $\eta<(\epsilon\delta)^{R}$ then we have  $h_m \equiv 0$.
\end{proposition}

\begin{proof} 
For any fixed $(a_0,\bar a)\in V^{m+1}$ consider the
system of affine forms in $(\bar y^0,\ldots, \bar y^m)$:
\[
(a_0  +\nu \cdot \bar a+  \omega \cdot \bar y^0+ \nu \cdot (\omega \cdot \bar y^1,\ldots, \omega \cdot \bar y^m))_{\bo \in \2^m \setminus \bar 0,\  \nu \in \2^m}
 \]
 By Lemma \ref{admisible} and Lemma \ref{counting} the set 
 \[
 B= \{(\bar y^0,\ldots, \bar y^m): (a_0  +\nu \cdot \bar a+  \omega \cdot \bar y^0+ \nu \cdot (\omega \cdot \bar y^1,\ldots, \omega \cdot \bar y^m)) \in X, \ \forall \bo \in \2^m \setminus \bar 0,\  \nu \in \2^m \} 
 \]
 is of size $$(\delta^{2^{m}\cdot (2^{m}-1)}+O(\eta))|V|^{m(m+1)}.$$

Consider the maps $\pi_{\nu}, p_{\omega}$ on $B$ 
\[\begin{aligned}
&\pi_{\nu}: (\bar y^0,\ldots, \bar y^m) \mapsto
(a_0  +\nu \cdot \bar a+  \omega \cdot \bar y^0+ \nu \cdot (\omega \cdot \bar y^1,\ldots, \omega \cdot \bar y^m))_{ \bo \in \2^m \setminus \bar 0} \\
&p_{\omega}: (\bar y^0,\ldots, \bar y^m) \mapsto 
 (a_0  +\nu \cdot \bar a+  \omega \cdot \bar y^0+ \nu \cdot (\omega \cdot \bar y^1,\ldots, \omega \cdot \bar y^m))_{\nu \in \2^m }
 \end{aligned}\]
 Consider $\nu=\bar 0$ and consider $\pi_{\bar 0}: B \to C_m(X)'$.  Fix a point $(a_0  + \omega \cdot \bar y^0)_{\bo \in \2^m}$ in the image and consider the system
 \[\begin{aligned}
&( \omega \cdot \bar y^0+ \nu \cdot (\omega \cdot \bar y^1,\ldots, \omega \cdot \bar y^m))_{\bo \in \2^m \setminus \bar 0,\  \nu \in \2^m}\\
&( \omega \cdot \bar y^0+ \nu \cdot (\omega \cdot \bar z^1,\ldots, \omega \cdot \bar z^m))_{\bo \in \2^m \setminus \bar 0,\  \nu \in \2^m}
\end{aligned} \]
By Lemma \ref{admisible} this system is of complexity $\le m$. It follows by Lemma \ref{counting} that $O(\delta^{-O(1)}\sqrt \eta)$  a.e. $(a_0  + \omega \cdot \bar y^0)_{\bo \in \2^m\setminus \bar 0} \in C_m(X)'$ we have 
$\pi^{-1}_{\bar 0}(a_0  + \omega \cdot \bar y^0)_{\bo \in \2^m\setminus \bar 0} \bigcap B$ is of size $(\delta^{C}+O(\eta^{1/4}))|V|^{m^2}$.  

 Since the function $h(a)=F_a(\bar v )$ for  $(O(\epsilon)+O(\delta^{-O(1)}\eta^{\Omega(1)}))$ a.e. $\bar v\in Y_a$, we get that $(O(\epsilon)+O(\delta^{-O(1)}\eta^{\Omega(1)}))$ a.e.
 $ (\bar y^0,\ldots, \bar y^m) \in B$ we have $h(a)=f_m'(\pi_{\bar 0}(\bar y^0,\ldots, \bar y^m))$.  Similarly for all $\nu$ it holds that $
 (O(\epsilon)+O(\delta^{-O(1)}\eta^{\Omega(1)}))$ a.e.
 $ (\bar y^0,\ldots, \bar y^m) \in B$ we have $h(a+\nu \cdot a)=f_m'(\pi_{\nu}(\bar y^0,\ldots, \bar y^m))$. \\

Now consider  $\bo=(1\bar 0)$, and the map  $p_{1\bar 0}: B \to C_m(X)$.  Fix a point 
$(a_0  +\nu \cdot \bar a+  y_1^0+ \nu \cdot (y_1^1,\ldots, y_1^m))_{\nu \in \2^m}$ in the image. 
Consider the system
\[\begin{aligned}
&( y^0_1+\omega \cdot (y_2^0, \ldots, y^0_m)+ \nu \cdot (y^1_1 +\omega \cdot (\bar y^1_2, \ldots, y^1_m),\ldots,y^m_1+ \omega \cdot (y^m_2, \ldots, y^m_m)))_{\bo \in \2^{m-1} \setminus \bar 0,\  \nu \in \2^m}\\
&( y^0_1+\omega \cdot (z_2^0, \ldots, z^0_m)+ \nu \cdot (y^1_1 +\omega \cdot (\bar z^1_2, \ldots, z^1_m),\ldots,y^m_1+ \omega \cdot  (z^m_2, \ldots, z^m_m)))_{\bo \in \2^{m-1} \setminus \bar 0,\  \nu \in \2^m}
\end{aligned} \]
By Lemma \ref{admisible} system is of complexity $\le m$. It follows by Lemma \ref{counting} that $O(\delta^{-O(1)}\sqrt \eta)$  a.e. $ y_1^0+ \nu \cdot (y_1^1,\ldots, y_1^m)$ we have 
$p^{-1}_{1\bar 0}((a_0  +\nu \cdot \bar a+  y_1^0+ \nu \cdot (y_1^1,\ldots, y_1^m))_{\nu \in \2^m}) \bigcap B$ is of size $(\delta^{C}+O(\eta^{1/4}))|V|^{m^2}$. 
 
Since $f_m$ vanishes $\epsilon$ -a.e. on $C_m(X)$ we get that for  $(O(\epsilon)+O(\delta^{-O(1)}\eta^{\Omega(1)})$ a.e $(\bar y^0,\ldots, \bar y^m) \in B$ we have 
$f_m(p_{1\bar 0}((\bar y^0,\ldots, \bar y^m))=0$. Similarly for other $\bo$,  so that $
( O(\epsilon)+O(\delta^{-O(1)}\eta^{\Omega(1)})$ a.e.
 $ (\bar y^0,\ldots, \bar y^m) \in B$ we have $f_m(\pi_{\bo}(\bar y^0,\ldots, \bar y^m))=0$. \\

  We obtain that for all $\nu\in \2^m, \bo \in \2^m \setminus \bar 0$ we have 
\[
f_m(p_{\bo}(\bar y^0,\ldots, \bar y^m))=0; \qquad h(a+\nu \cdot a)=f_m'(\pi_{\nu}(\bar y^0,\ldots, \bar y^m))=0
\]
for $(O(\epsilon)+O(\delta^{-O(1)}\eta^{\Omega(1)}))$-a.e.  $(\bar y^0,\ldots, \bar y^m) \in B$. Thus there is a constant $R, \alpha>0$ so that for $\epsilon<\alpha$,   $\eta<\delta^R$ we have that  this set is not empty and thus $h_m(a|\bar a) =0$.
  \end{proof}

Finally we claim that there is a choice of $\alpha, B>0$ so that for $\ep < \alpha$  and $\eta<(\epsilon\delta)^{B}$we have   $h(a)=f(a)$ for $O(\epsilon)$ a.e.   $a \in X$ :
 On the one hand we can choose such $\alpha, B,C$ so that we have $h(a)=F_a(\bar v)$ for $C\epsilon$ a.e. $\bar v \in Y_a$.   Choose $\epsilon$ so that also $C\epsilon < 1/4$. 
 
 On the other hand,
let $$Y=\{(a,\bar v): (a|\bar v) \in C_m(X)\}$$
Then $Y = \bigcup_{a \in X} (a,Y_a)$. By Lemma \ref{size} $|Y_a|=(\delta^{2^{m}-1}+O(\eta))|V|^m$  for all $a$. By Lemma \ref{fubini}  since  $f_m \equiv 0$ for  
$\epsilon$-a.e. $(a|\bar v) \in C_m(X)$ then for  $8\epsilon$ a.e. $a$, for $(1+O(\delta^{-O(1)}\eta))^2/8$ a.e. $\bar v \in Y_a$, we have $f(a)=F_a(\bar v)$. Choose $\eta$ so that  so that $(1+O(\delta^{-O(1)}\eta))^2/8 < 1/4$. Then for $O(\epsilon)$ a.e. $a \in X$ we can find $\bar v \in Y_a$ such that 
such that $h(a)=F_a(\bar v)$ and $f(a)=F_a(\bar v)$, so that $f(a)=h(a)$.

\section{Splining result implies extension result}\label{follows}

In this section we show how to deduce Theorem \ref{extending-uniform} from Theorem \ref{testing-uniform}.\\

Let $X \subset V$ be of density $\delta>0$. Then by Theorem 
\ref{testing-uniform} there exist $\alpha, B ,C$ such that if  $0<\epsilon <\alpha$ and $\eta < (\epsilon\delta)^{C}$ and  $X$ is $(\eta,m)$-uniform, and $f_m|_{X} \equiv 0$, then we can find $h:V \to H$ such that $h_m\equiv 0$, and  $h (x)=f(x)$ on $C \epsilon$ a.e. $x \in X$. We claim
that for  $\epsilon$ sufficiently small we will have $h|_X \equiv f$. \\

Consider the function $g:X \to H$ defined by $g=f-h|_X$.  Then $g$ vanishes $C \epsilon$-a.e, $x \in X$  and $g_m$ vanishes on $C_m(X)$.  Let $Z$ be the set of  $z$ with $g(z)=0$. \\

For any $x\in X$ recall that
$$Y_x=\{\bar v: (x|\bar v) \in C_m(X)\}. $$

By Lemma \ref{size}, if $X \subset V$ is of density $\delta$, and $(\eta, m)$-uniform then
$$|Y_x|= (\delta^{2^{m}-1}+O(\eta))|V|^{m}.$$  
Namely all $x \in V$ participate in many $m$-cubes in $X$. 

Denote by $R_x$ the corresponding set of cubes completions:
\[
R_x=\{(x| \bar v)': \bar v \in Y_x\}.
\] 
Then $|R_x|= (\delta^{2^{m}-1}+O(\eta))|V|^{m}$.
We count how many $(x| \bar v)' \in R_x$ do not have a point in $Z$:  we can estimate the  number of cubes in $R_x$ with $x+v_1$ in the bad set (squared):
\[\begin{aligned}
&|\mE_{\bar v } 1_{Z^c \cap X}(x+v_1)\prod_{\bo \in \2^m\setminus \{\bar 0\}} 1_X(x+\omega \cdot \bar v)|^2  \le ( \mE_{v_1  } 1_{Z^c \cap X} (x+v_1))\mE_{v_1} |\mE_{v_2, \ldots, v_m }\prod_{\bo \in \2^m\setminus \{\bar 0\}} 1_X(x+\omega \cdot \bar v)|^2\end{aligned}\] 
which is bounded by Lemma \ref{counting} by $ C\epsilon (\delta+O(\eta)) (\delta^{2^{m+1}-2}+O(\eta))$, so that for $\eta <  O(\delta^{2^{m+1}-2})$ the  number of cubes in $R_x$ with $x+v_1$ in the bad set is $O(  \epsilon)|R_x|$, and similarly for other vertices in $(x|\bar v)'$).  Thus  for  
$\ep$ sufficiently small the set of  cubes completions of $x$ is with all vertices in $Z$ is not empty, i.e 
 for any $x \in X$ we can find $\bar v$, so that $(x| \bar v)$ is an $m$-cube and $x+\omega\cdot \bar v \in Z$ for $\omega \ne \bar 0$. Since  $g_m$ vanishes on $C_m(X)$,
 we get $g(x)=0$. Thus $h$ is an extension of $f|_X$. \\

%%%%%%%%%%%%%%%%%%%%%%%%%%%%%%%%%%%

\section{Polynomial splining on subvarieties}\label{testing}

In this section $V$ is a vectors space over a finite field $k=\mF_q$. and $X$ a subvariety of degree $\le d$ which is a complete intersection and codimension $L$.  We prove a  general splining statement about functions from $X \to H$, $H$ some abelian group.  This proposition is of independent interest, and it does not require $X$ to be of high rank. In this section $O(1)$ is a constant depending on $d, L$ and the dimension of the cubes $m$. We will suppress this dependence.

In this section we prove Theorem \ref{testing-X-intro}. Theorem \ref{testing-subspace-X-intro} can be derived in a similar manner adapting to proof of Theorem \ref{testing-X-intro} below to the arguments in \cite{KR}.

The key ingredient is the following proposition which we prove in the appendix.
 \begin{proposition}\label{solutions-X}For any triple $d,m,L\geq 1$ there exists $n=n(d,m,L)$  such that for any $k$-vector space $V$, homogeneous polynomials    $\{P_i:V\to k\}_{i=1}^L$ of degrees  $\le d$ and points 
$a_j\in X,1\leq j\leq m, X:=\{x:  P_i(x)=0 , i=1, \ldots, L\}$
 we have $|Y|\geq q^{-n}|V|$
where 
$$Y=\{x\in V: P_i(x+a_j)=0,1\leq i\leq L,1\leq j\leq m\}.$$ 
 \end{proposition}

 \begin{remark} $n=n(d,m,L)$ depends linearly on $L, m$.
 \end{remark}

 \begin{corollary}\label{solutions-X1}For any triple $d,m,L\geq 1$ there exists $n=n(d,m,L)$ such that for any $k$-vector space $V$,  polynomials    $\{P_i:V\to k\}_{i=1}^L$ of degrees  $\le d$ and points 
$a_j\in X,1\leq j\leq m, X:=\{x: P_i(x)=0\}$
 we have $|Y|\geq q^{-n}|V|$
where 
$$Y=\{x\in V: P_i(x+a_j)=0,1\leq i\leq L,1\leq j\leq m\}.$$ 
 \end{corollary}
 
\begin{proof} By Lemma \ref{b}. 
\end{proof}

\begin{proof}[Proof proposition \ref{testing-X-intro}]
Let $X$ be a variety of degree $d$ and codimension $L$ that is a complete intersection. Let $f:X \to H$ s.t. $f_m$ vanishes 
$\epsilon$-a.e $c \in C_m(X)$. 

For $a \in X$ define
\[
Y_a=\{\bar v: (x|\bar v)' \in C_m'(X)\}
\]
By Corollary \ref{solutions-X1}  for any $a \in X$ we have $|Y_a| =q^{-O(1)}|V|^m$ \footnote{For sufficiently high rank $Y_a$ is approximately of size $q^{-(2^{m}-1)L}|V|^m$, but in this section we don't have any rank assumptions.}.
Define 
\[
F_a(\bar v)= f_m(a|\bar v)-f(a)
\]
\begin{lemma}\label{variety-constant} $F_a(\bar v)$ is constant $q^{O(1)}\epsilon$ a.e. $v \in Y_a$.
\end{lemma}

\begin{proof}
As in Proposition \ref{constant} we have 
\[\begin{aligned}
F_a(\bar v )-F_a(\bar u )&= 
 \sum_{i=1}^m  f_m(a+u_i|w_1,\ldots, w_{i-1}, w_i-u_i,u_{i+1} \ldots, u_m)\\
& \quad - \sum_{i=1}^m f_m (a+v_i|w_1,\ldots, w_{i-1}, w_i-v_i,v_{i+1} \ldots, v_m).
\end{aligned}\]

Let $S$ be the set of $(\bar v, \bar u, \bar w) \in Y_a^2 \times V^m$ such that for all $i=1, \ldots, m$ we have
\[
 (a+u_i|w_1,\ldots, w_{i-1}, w_i-u_i,u_{i+1} \ldots, u_m),  (a+v_i|w_1,\ldots, w_{i-1}, w_i-v_i,v_{i+1} \ldots, v_m) \in C_m(X).
\]
By Corollary \ref{solutions-X1} $|S|=q^{-O(1)}|V|^{3m}$.

Consider the maps: $\pi_i, p_i: S  \to V^{2^m}$ defined by
\[\begin{aligned}
&\pi_i:(\bar v, \bar u ,\bar w )\mapsto  (a+u_i|w_1,\ldots, w_{i-1}, w_i-u_i,u_{i+1} \ldots, u_m),  \\
&p_i:(\bar v, \bar u ,\bar w )\mapsto  (a+v_i|w_1,\ldots, w_{i-1}, w_i-v_i,v_{i+1} \ldots, v_m).
\end{aligned}
\]

Fix  $(s|\bar t) \in C_m(X)$ for which the preimage under $\pi_i$ of $(s|t)$ is not empty,  This oreimage is the set of  $(\bar v, \bar u ,\bar w )$ such that 
\[\begin{aligned}
&(a|\bar v)', (a|\bar u)'  \in C_m'(X), \\
&  a+u_i=s, w_1=t_1,\ldots, w_{i-1}=t_{i-1}, w_i-u_i=t_i,u_{i+1}=t_{i+1} \ldots, u_m=t_m
\end{aligned}\]
which is the same as
\[\begin{aligned}
&(a+\nu \cdot \bar v)_{\nu \in \2^m} \in X, (a+\nu \cdot (u_1, \ldots, u_{i-1}, s-a,t_{i+1}, \ldots, t_m ))_{\nu \in \2^m} \in X, \\
&w_1=t_1,\ldots, w_{i-1}=t_{i-1}, w_i-u_i=t_i.
\end{aligned}
\]
Observe that the condition $(a+\nu \cdot (u_1, \ldots, u_{i-1}, s-a,t_{i+1}, \ldots, t_m ))_{\nu \in \2^m} \in X$ can be written as
\[
(a+\nu \cdot (u_1, \ldots, u_{i-1}, t_{i+1}, \ldots, t_m ) \in X; \  
(s+\nu \cdot (u_1, \ldots, u_{i-1}, t_{i+1}, \ldots, t_m );  \quad \nu \in \2^{m-1} 
\]
Since the preimage is not empty,  we have in particular  that $a+\nu'  \cdot( t_{i+1}, \ldots, t_m ) \in X$ for any $\nu' \in \2^{m-i}$. 
By Corollary \ref{solutions-X1} this system has at least  $q^{-O(1)}|V|^{2m-1}$ many solutions and clearly it has at most $|V|^{2m-1}$ solutions. By Lemma \ref{ae-projection1}, since $f_m$ vanishes 
$\epsilon$-a.e. on $C_m(X)$ we find that $f_m(\pi_i(\bar v, \bar u ,\bar w))$ vanishes for   $q^{O(1)}\epsilon$-a.e $(\bar v, \bar u, \bar w) \in S$

 Similarly for the maps $p_i$ we have $f_m(p_i(\bar v, \bar u ,\bar w))$ vanishes for $q^{O(1)}\epsilon$-a.e $(\bar v, \bar u, \bar w) \in S$.  It follows that for $q^{O(1)}\epsilon$-a.e $(\bar v, \bar u, \bar w) \in S$ we have 
 $f_m(\pi_i(\bar v, \bar u, \bar w) )=f_m(p_i(\bar v, \bar u, \bar w)) =0$ for $i=1, \ldots, m$.  \\

 For $\bar v, \bar u \in Y_a$ let $S_{\bar u, \bar v}$ be the set of $\bar w \in V^m$ such that $(\bar v, \bar u, \bar w)  \in S$. By Corollary \ref{solutions-X1} $|S_{\bar u, \bar v}|$ is of size $q^{-O(1)}|V|^m$, so is not empty for $|V|$ sufficiently large.  Now $S=\bigcup _{\bar u, \bar v} (\bar u, \bar v, S_{\bar u ,\bar v})$. It follows that for $q^{O(1)}\epsilon$ a.e.
$\bar v, \bar u \in Y_a$ we can find $\bar w$ such that   $f_m(\pi_i(\bar v, \bar u, \bar w) )=f_m(p_i(\bar v, \bar u, \bar w)) =0$ for $i=1, \ldots, m$, and thus $F_a(\bar v)-F_a(\bar w) =0$. 
\end{proof}

We define $h: X \to H$ setting $h(x)$ to be the common value of $F_a(\bar v)$, as $\bar v$ range over  $Y_a$.

\begin{lemma} For $\epsilon$ sufficiently small in terms of $d, L, m$, we have $h_m|_{X} \equiv 0$.
\end{lemma}
\begin{proof}  
For any fixed $(a_0,\bar a)\in V^{m+1}$
system of affine  forms in $(\bar y^0,\ldots, \bar y^m)$ 
\[
(a_0  +\nu \cdot \bar a+  \omega \cdot \bar y^0+ \nu \cdot (\omega \cdot \bar y^1,\ldots, \omega \cdot \bar y^m))_{\bo \in \2^m \setminus \bar 0,\  \nu \in \2^m}
 \]
 By Corollary \ref{solutions-X1} the set
 \[
 B= \{(\bar y^0,\ldots, \bar y^m): (a_0  +\nu \cdot \bar a+  \omega \cdot \bar y^0+ \nu \cdot (\omega \cdot \bar y^1,\ldots, \omega \cdot \bar y^m)) \in X, \ \forall \bo \in \2^m \setminus \bar 0,\  \nu \in \2^m \} 
 \]
 is of size  $q^{-O(1)}|V|^{m(m+1)}$.\\

Consider the maps $\pi_{\nu} : V^{m(m+1)} \to V^{2^m-1}$ and $p_{\omega} : V^{m(m+1)} \to V^{2^m}$  defined by
\[\begin{aligned}
&\pi_{\nu}: (\bar y^0,\ldots, \bar y^m) \mapsto
(a_0  +\nu \cdot \bar a+  \omega \cdot \bar y^0+ \nu \cdot (\omega \cdot \bar y^1,\ldots, \omega \cdot \bar y^m))_{ \bo \in \2^m \setminus \bar 0} \\
&p_{\omega}: (\bar y^0,\ldots, \bar y^m) \mapsto 
 (a_0  +\nu \cdot \bar a+  \omega \cdot \bar y^0+ \nu \cdot (\omega \cdot \bar y^1,\ldots, \omega \cdot \bar y^m))_{\nu \in \2^m }
 \end{aligned}\]
$\pi_{\nu}, p_{\omega}$ are linear maps on $V^{m(m+1)}$ thus all fibers are of the same size.

By Proposition \ref{variety-constant} the function $F_a(\bar v )$ is constant  $q^{O(1)}\epsilon$  a.e. $\bar v \in V^m$, so that
 for $q^{O(1)}\epsilon$ a.s.
 $(\bar y^0,\ldots, \bar y^m) \in V^{m(m+1)}$, such that $\pi_{\nu}(\bar y^0,\ldots, \bar y^m) \in C_m(X)'$ we have 
 \[
 h(a+\nu \cdot \bar a)= f'_m( \pi_{\nu}(\bar y^0,\ldots, \bar y^m)).
  \]

Since $f_m$ is zero $\epsilon$-a.e it follows that for $\epsilon$-a.e $(\bar y^0,\ldots, \bar y^m) \in V^{m(m+1)}$,
\[
f_m( p_{\bo}(\bar y^0,\ldots, \bar y^m))=0.
\]

 It follows that  $q^{O(1)}\epsilon$  a.e $(\bar y^0,\ldots, \bar y^m) \in B$ we have for all $\nu, \bo$
 \[
 f_m( p_{\bo}(\bar y^0,\ldots, \bar y^m))=0, \quad  h(a+\nu \cdot \bar a)=f'_m(\pi_{\nu}(\bar y^0,\ldots, \bar y^m) ),
 \]
 so that $h_m(a|\bar a) =0$.
\end{proof}

Finally we need to show that $q^{O(1)}\epsilon$ a.e. $x \in X$ we have $h(x)=f(x)$.  Now $C_m(X)=\bigcup_{x \in X} (x, Y_x)$, and $f_m$ vanishes $\epsilon$ a.e. on $C_m(X)$ thus  by Lemma \ref{fubini}  for  $q^{O(1)}\epsilon$ a.e. $\bar v \in Y_x$  we have $h(x)=F_x(\bar v)$ and $q^{O(1)}\epsilon$ a.e. $x$, for  $\bar v \in Y_x$  we have $f(x)=F_x(\bar v)$.

\end{proof}

\section{Polynomial splining  on subvarieties; high rank case}

In this section we prove Theorem \ref{testing-X-high}. Let $V$ be a vector space over a finite field $k$.
We assume that $k=\mF _q$,  $q=p^l$ and denote by $e_q:k\to \mC^\star$ the additive character 
$e_q (a):=\exp (tr _{k/\mF _p}(a))$.
 Let $\bar P=(P_{ij}: 1 \le j \le d, 1 \le i \le M_j)$ be a collection of polynomials,  where $\{P_{ij}\}_{i=1}^{M_j}$ are of degree exactly $j$, 
$P_{ij}:V \to k$. Let $L=\sum M_j$. We denote by   
$\mathcal B$ the  level sets of the polynomials $\bar P$.  Let $\Sigma = \prod_{j \in [d]} k^{M_j}$.  For any $\bar a \in \Sigma$ denote 
$X_{\bar a}$ the variety $\bar P = \bar a$.  Restricting to a subspace of codimension bounded by $L$, we may  assume all polynomials in $\bar P$ are of degree $\ge 2$.
 All bounds $O(1), \Omega(1)$ in this section depend on $d,L,k$; we suppress this dependence.

Fix $\bar a$ and let $X=X_{\bar a}$ (High rank implies that $X_{\bar a}$ are all of essentially the same size).  

We assume that $\bar a =0$ for simplicity in notation; the proof is the same for other $\bar a$. We follow the proof in section $4$, with two observations: the first is that all sets in the proof in section $4$ have many points as long as $a \in X$ (as in previous section, this is not a high rank property; high rank gives a precise estimate, but this is not necessary for the argument). The  second observation (Lemmas \ref{quadratic}, \ref{general} below) that in the high rank case all maps in the proofs of the various Lemmas in section $4$ have fibers of essentially the same size; in section $4$ we this property was obtained using the uniformity condition.

We consider first the case $d=2$, and let $P$ be of degree $2$.  Denote by $(x,y)$ the bilinear form $P(x+y)-P(x)-P(y)$. 
For $ n \in \mN$  denote $[n]_*^2=\{(i,j):  1 \le i \le j \le n\}$.
The key is the following Lemma: 

\begin{lemma}\label{quadratic}  Let $n, m \in \mN$.  Let $P$ be of degree $d$ and rank $>r$. Let  $J \subset [m]_*^2$,  $K \subset [n]_*^2$, $I\subset [m]_* \times [n]_*$. 
 For $q^{-\Omega_{m,n}(r)}$-almost any $\bar t \in V^{m}$ such that $(t_i, t_j) =0$ for  $(i,j) \in J$,
   the system in $\bar v \in V^n$.
\[
(v_i, v_j) =0; (i,j) \in K;   \quad   (v_i, t_j) =0; (i,j) \in I
\]
has the $(q^{-C}+O(q^{-\Omega_{m,n}(r)}))|V|$ many solutions, where $C$ is the sum of the sizes of the non empty of the set $I$, $J$, $K$.
\end{lemma}

\begin{proof}
For fixed $\bar t$, the number of solutions is given by
\[
q^{-(|K|+|J|+|I|)}\sum _{\bar a \in  k^I, \bar b \in  k^J, \bar c  \in  k^K} \sum_{\bar v} e_q(\sum_{(i,j) \in I} \bar a_{i,j} (v_i, t_j) +\sum_{(i,j) \in J}\bar b_{i,j}(t_i, t_j) +\sum _{(i,j) \in K}\bar c_{i,j}(v_i, v_j) ).
\]
 The contribution to the sum of  $(\bar a, \bar b, \bar c) = \bar 0$ is $q^{-(|K|+|J||+I|)}$. 
 
For fixed $(\bar a, \bar b, \bar c) \ne \bar 0$ consider the average
\[\begin{aligned}
&\mE_{\bar t}| \mE_{\bar v} e_q(\sum_{(i,j) \in I} \bar a_{i,j} (v_i, t_j) +\sum_{(i,j) \in J}\bar b_{i,j}(t_i, t_j) +\sum _{(i,j) \in K}\bar c_{i,j}(v_i, v_j) )|^2\\
&= \mE_{\bar t} \mE_{\bar v, \bar v'} e_q(\sum_{(i,j) \in I} \bar a_{i,j} (v'_i, t_j) +\sum _{(i,j) \in K}\bar c_{i,j}((v_i, v_j')+(v_i', v_j)+(v'_i, v'_j) )
\end{aligned}\]
Since $P$ is of rank $>r$ the contribution is $q^{-r/2}$.   Now the Lemma follows from Lemma \ref{fubini}. 
\end{proof}

As mentioned above, the proof of Theorem \ref{testing-X-high} follow along the lines of the proof of  Theorem \ref{testing-uniform}, except the uniformity in the sizes of the fibers of the various maps in the proof now follows from the high rank property, instead of uniformity. 

We demonstrate this with a proof of Lemma \ref{A_m} in the case $d=2$ under the condition of  high rank (replacing the condition of uniformity). 
We consider the case $L=1$; the case $L>1$ is similar. 

\begin{lemma}\label{QA_m}
Consider the set
\[
A_m=\{ u_m, v_m, w_1, \ldots, w_m  : (a+u_m|w_1,\ldots, ,w_{m-1}, w_m-u_m), (a+v_m|w_1,\ldots, w_{m-1}, w_m-v_m) \in C_m(X)\}.
\]
Then for any $s>0$ there exists $r=r(s,d,m)$ such that   for $(O(\epsilon)+O(q^{-s}))$ a.e. $(u_m, v_m , \bar w) \in A_m$ we have
$$f_m(a+u_m|w_1,\ldots, w_{m}-u_m) =f_m(a+v_m|w_1,\ldots, w_{m}-v_m) =0.$$ 
\end{lemma}

\begin{proof}
Consider the map $p_m:V^{m+2} \to C_m(V)$ defined by  
\[
(u_m, v_m , \bar w) \mapsto (a+u_m|w_1,\ldots, ,w_{m-1}, w_m-u_m).
\] Fix a cube in 
$(x|\bar t) \in C_m(X)$ and consider the intersection $p_m^{-1}((x|\bar t)) \cap A_m$. The proof  of Lemma \ref{A_m} is based on showing that the sizes of these fibers are almost surely of essentially the same size. 

Consider the  intersection $p_m^{-1}((s|\bar t)) \cap A_m$. This is the set  of $u_m, v_m, w_1, \ldots, w_m$ such that
\[
 (a+u_m|w_1,\ldots, ,w_{m-1}, w_m-u_m), (a+v_m|w_1,\ldots, w_{m-1}, w_m-v_m)  \in X
\]
and 
\[
 (a+u_m|w_1,\ldots, ,w_{m-1}, w_m-u_m)=(s| \bar t). 
\]
Namely we are looking for solutions to the system
\[
P(a+v_m+\nu \cdot (w_1,\ldots, w_{m-1}, w_m-v_m))=0;  P(a+u_m+\nu \cdot (w_1,\ldots, w_{m-1}, w_m-u_m))=0; \
 \nu \in \{0,1\}^m\]
given 
\[
a+u_m=s, w_1=t_1, \ldots, w_{m-1}=t_{m-1}, w_m-u_m = w_m-(x-a)= t_m.
\]
Namely this is the set of $v_m$ such that 
\[
P(a+v_m+\nu \cdot (t_1,\ldots, t_{m-1}, t_m+(x-a)-v_m))=0; \ \nu \in \{0,1\}^m,
\]
which is the same as finding solutions to 
\[
P(v_m+\nu \cdot (t_1,\ldots, t_{m-1}, t_m+x-v_m))=0;   \nu \in \{0,1\}^m.
\]
This translates to the system of equations in $v_m$
\[
(v_m,v_m)=0;\  (v_m, t_j)=0,  j=1, \ldots, m-1; \ (v_m, t_m+x-v_m)=0
\]
 By Lemma \ref{quadratic} there exists $C>0$ such that we can choose $r$ so that for $q^{-s}$ a.e. $(x|\bar t) \in C_{m}(X)$ this system has  $(q^{-C}+O(q^{-s}))|V|$ many solutions.  Thus we can choose $r$ so that for  $q^{-s}$ a.e. $(x|\bar t) \in C_m(X)$ we have that
$p_m^{-1}((x|\bar t)) \cap A_m$ is of size $(q^{-C}+q^{-s})|V|$.  
\end{proof}

The rest of the Lemmas in section $4$ follow from Lemma \ref{quadratic} in the same way. \\
\ \\
We turn to the case $d>2$. If $P$ is of degree $d$ then denote 
\[
(v_1, \ldots, v_d) = P_{d}(x|v_1, \ldots, v_d).
\]
This is a multilinear symmetric  form (and is independent of $x$).  For $n \in \mN$ denote $[n]^k_*$ the set 
$\{(n_1, \ldots, n_k): 1 \le n_i \le n,  i<j \implies n_i \le n_j \}$.  

The key Lemma here is the following  analogue of Lemma \ref{quadratic}; Theorem  \ref{testing-X-high} will from Lemma  \ref{general} in the same way:

\begin{lemma}\label{general}  Let $s>0$, Let $m, n \in \mN$. There exists $r=r(s,d, k,m,n)$, such that for any $P$ be of degree $d$ and rank $>r$, any 
Let $J_i \subset  [m]_*^{i} \times [n]_*^{d-i}$ for $i=0, \ldots, d$, the following holds: 
for $p^{-s}$ almost any $\bar t \in V^{m}$ the system in $\bar v \in V^n$ 
\[
(v_{j_1}, \ldots, v_{j_i}, t_{l_1}, \ldots, t_{j_{d-i}}) =0; \ i=0, \ldots, d; \ (j_1, \ldots, j_i,l_1, \ldots, l_{d-i}) \in J_i 
\]
has the $(q^{-C}+O(q^{-s}))|V|$ many solutions, where $C=\sum_{i: J_i \ne \emptyset} |J_i|$.  
\end{lemma}

\begin{remark} Lemma \ref{general} holds for any system of polynomials with the same proof for a family of polynomials  $\bar P$ of degree $\le d$ and rank $>r$ (one needs length $j$ multilinear forms for the polynomials of degree $j$ in the collection). 
\end{remark}

 Lemma \ref{general} follows from the following Theorem:
\begin{theorem}[\cite{gt1}, \cite{kl}, \cite{bl}]\label{equi} Let $s>0$. There exits $r=r(s,d,k)$ so that for any $P$ of degree $d$ and rank $>r$ 
\[
|\mE_{x_1, \ldots, x_d} e_q(x_1, \ldots, x_d) | \le q^{-s}.
\] 
 Furthermore $r=r(k,d,s)=r(d,s)$ for all $k$ of char $>d$.
\end{theorem}

\begin{proof}[Proof of Lemma \ref{general}]
For fixed $\bar t$, the number of solutions is given by
\[
q^{-C} \sum_{0 \le i \le d, \bar b^i \in k^{J_i}} \sum_{v}  e_q(\sum_i \sum_{\bar j \in J_i} \bar b^i_{\bar j}((v_{j_1}, \ldots, v_{j_i}, t_{l_1}, \ldots, t_{j_{d-i}}).
\]
 The contribution to the sum of  $(\bar b^1, \ldots, \bar b^d)= \bar 0$ is $q^{-C}$.
 
 We fix $(\bar b^1, \ldots, \bar b^d) \ne \bar 0$, and evaluate the average
\[
\mE_{\bar t \in V^m}| \mE_{\bar v}  e_q(\sum_{1 \le i \le d, \bar j \in J_i} \bar b^i_{\bar j}((v_{j_1}, \ldots, v_{j_i}, t_{l_1}, \ldots, t_{j_{d-i}}).|^2
\]
which is 
\[
\mE_{\bar t \in V^m}  \mE_{\bar v, \bar v'}  e_q\big(\sum_{1 \le i \le d, \bar j \in J_i} \bar b^i_{\bar j}
((v_{j_1}, \ldots, v_{j_i}, t_{l_1}, \ldots, t_{j_{d-i}}) - ((v'_{j_1}, \ldots, v'_{j_i}, t_{l_1}, \ldots, t_{j_{d-i}})\big).
\]

After repeated applications of Cauchy-Schwartz with an expression of the from 
\[
\mE_{x_1, \ldots, x_d} e_q(c(x_1, \ldots, x_d) )
\] 
for some $c \ne 0$. By Theorem \ref{equi} this is $\le q^{-s}$ for sufficiently large rank $r$.  Now the Lemma follows from  Lemma \ref{fubini}. 

To demonstrate how this is done: choose a variable appearing non trivially. Assume without loss of generality this is $t_1$. 
Let $h$ be the number of appearances. After $h-1$ applications of the Cauchy Schwartz inequality to the above average, isolating only the expressions containing $t_1$, we obtain in the exponent a sum of terms in $t_1,t^1_1, \ldots, t^{h-1}_1, t_2, \ldots, t_m, \bar v, \bar v'$, and the dependence on $t_1$ is linear. We can do this for all variables, so we may assume that at the expense of adding some variables, all terms are in all variables. In particular any term
contains a unique set of parameters, and is linear in each of them. Say for example the term is $(v_1, t_2, \ldots, t_d)$. After $1$ application of Cauchy-Schwartz, isolating the terms containing $v_1$, we are left only with terms that contain $v_1$. Repeating this, after $d-1$ more applications we are left with the terms containing only $v_1,t_2 \ldots, t_d$ but there is a unique such term - $(v_1, t_2, \ldots, t_d)$.  
\end{proof}

Theorem  \ref{testing-X-high} follows from Lemma  \ref{general} in the same way:  as an example we demonstrate the proof of  Lemma \ref{QA_m} for $d>2$.  By Lemma \ref{general} there exists $C=C(m)$ for any $s>0$ there exists $r=r(s,d, L)$ such that  for $q^{-s}$ a.e $\bar t \in V^{m}$ such that $P_i(a+\nu \cdot (t_1,\ldots, t_{m}))=0$, $\nu \in \{0,1\}^m$, $i=1, \ldots, L$ the system  of equations in $v_m$
\[
P_i(a+v_m+\nu \cdot (t_1,\ldots, t_{m}))=0; \ \nu \in \{0,1\}^m, i=1, \ldots, L
\]
has $|V|(q^{-CL}+O(q^{-s}))$ many solutions.

\appendix \label{appendix-solutions}

\section{Subvarieties of bounded degree and codimension in high dimensional vector spaces contain many lines}

Let $k=\mF _q$, $V$ be a $k$-vector space, $N=dim(V)$, and $\mP(V)$ the corresponding projective space. For any subspace $W\subset V$ we have a natural embedding $\mP (W) \ho \mP (V)$.  \\

\begin{definition} A subset  $ Z\subset \mP (V)$ is {\it $D$-large} if $Z\cap \mP (W)\neq \emp$ for any subspace $W$ of $V$ of dimension $> D$.
\end{definition}

\begin{lemma}\label{large1}  $|Z|\geq |\mP (V)|/2q^{D+1}$ for any {\it $D$-large} subset 
$Z\subset \mP (V)$.  
\end{lemma}

\begin{proof}
Let $Gr_N^D$ be the set of $D+1$-dimensional subspaces $W$ of $V$. It is well known that $|Gr_N^D| =\binom{N}{D+1}_q= \frac {\prod _{i=0}^d(1-q^{N-i})}{\prod _{i=1}^{D+1}(1-q^i)}$.
This obviously imply the inequality
 $$|Gr_N^D|/|Gr _{N-1}^{D-1} |\geq \frac{q^{N}-1}{q^{D+1}-1} \ge \frac{q^{N}}{2q^{D+1}}\cdot \frac{q}{q-1}$$

For any $l\in \mP (V)$ we denote  by  $Gr_l\subset Gr_N^D$ the set of $D+1$-planes $W$ containing $l$. The size of  $Gr _l$ does not depend on $l$, and is equal to  $|Gr_{N-1}^{D-1}|$. Since $Z$ is {\it $D$-large}, we know that for  any $D$-dimensional subspace $W$ of $V$ there exists $l\in Z$ such that $W\in Gr _l$. So  $Gr _N^D=\bigcup _{l\in Z}Gr _l$ and therefore $ |Gr_N^D| \leq |Z| |Gr _{N-1}^{D-1}| $. So 
$$|Z|\geq |Gr_N^D|/|Gr _{N-1}^{D-1}| \ge \frac{q^N}{2q^{D+1}}\frac{q}{q-1}.$$ 
Since $|\mP (V)|/q^N\leq \frac{q}{q-1}$ the Lemma is proven.
\end{proof}

Let   $V$ be a $k$-vector space,
 $P_i:W\to k,1\leq i\leq s$ be homogeneous  polynomials of  degrees $d_i\ge 1$ and let  $D:=\sum _id_i$.  Let  $\ti Y=\{ v|P_i(v)=0\}, 1\leq i\leq s$. The subset $\ti Y$ of $V$ is homogeneous. We denote by $Y$ the corresponding subset of $\mP (V)$.

\begin{corollary}\label{large} $|Y(k)|\geq |\mP (V)|/2q^{D+1}$.
\end{corollary}
\begin{proof}As follows from Lemma \ref{large1} it is sufficient to show that  $Y$ is {\it $D$-large} that is that 
$Y\cap W\neq \{ 0\}$ for any $D$-dimensional subspace $W$ of $V$.
In other words we have to show that there exists a non-zero $w\in W,w\neq 0$ such that $P_i(w)=0, 1\leq i\leq s$. But the existsence of such $w$ follows from Corollary to the main theorem in the end of section $3$ of \cite{ax}.
\end{proof}

Let $P:V\to k$ be a homogeneous polynomial of degree $d \ge 1$ and 
$X:=\{ v|P(v)=0\}$. Fix $x_j \in X \setminus 0$ , $j=1, \ldots, m$ and define 
 $Y'$ as  the set of 
$y\in X$ such that  $y+tx_j\in X(\bar k)$ for all $t\in \bar k$, $j=1, \ldots, m$. 

\begin{lemma}\label{many}
 $|Y'(k)|\geq |\mP (V)(k)|/2q^{md(d+1)/2 +1}$.
\end{lemma}

\begin{proof}Let $D:=
md(d+1)/2 $.
We expand
\[
 P(tx_j+y) =\sum t^iP_{ij}(y)
 \]
 where $P_{ij}(y)$  are homonogeneous polynomials of degree $d-i$. Let
 $Y=\{v| P_{ij}(y)=0\}, 0\leq i<d$, $j=1, \ldots, m$. 
Then  $Y\subset Y'$. By Corollary  \ref{large}  we have $|Y|\geq |V|/2q^{D+1}$.
\end{proof}

Let $P_i :V\to k$, with $i =1, \ldots, c$ be  homogeneous polynomials of degrees $1 \le d_ i \le d$, let  $\bar d=(d_1, \dots ,d_c)$, 
and let  $X:= \bigcap_i \{ v|P_i (v)=0\}$. Fix $x_j \in X \setminus 0$ , $j=1, \ldots, m$ and define 
 $Y'$ as  the set of 
$y\in X$ such that  $y+tx_j\in X(\bar k)$ for all $t\in \bar k$, $j=1, \ldots, m$. 

\begin{lemma}\label{analog}
 $|Y'(k)|\geq |\mP (V)(k)|/q^{O_{\bar d, m}(1)}$.
\end{lemma}
The proof is completely analogous to the proof of Lemma \ref{many}.

\begin{lemma}\label{b} For a family $\bar P =(P_1, \dots ,P_c) $
of polynomials $P_i$ on a $k$-vector space $V$ we define $X_{\bar P}=\{ x\in V|P_i(x)=0\},1\leq i\leq c$.
For any degree vector $\bar d=(d_1, \dots ,d_c)$ there exists a constant 
$C(\bar d)$ such that the following holds. For any finite field $k=\mF _q$, $q>d:=\max _id_i$ 
and a point $x\in X_{\bar P}$ we have $|\mcL (x)|\geq q^{dim(V)-C(\bar d)}$ where 
$\mcL (x)$ is the set of affine lines $L\subset X_{\bar P}$ containing $x$.
\end{lemma}
\begin{proof}Let $a_0,\dots ,a_d$ be distinct elements of $k$. For any $v\in V-0$ denote by  by 
$L_v$ the image of the map $t\to x+tv,t\in k$. It is clear that $L_v\in \mcL$ for any $v\in V-0$ 
such that 
$P_i(x+a_jv)=0,1\leq i\leq c,0\leq j\leq d$.

Define polynomials $P_{i,j}$ on $V$ by 
$P_{i,j}(v)= P_i(x+a_jv)$. Since $P_{i,j} (0)=0$
 we can  write them as sums of homogeneous polynomials 
$P_{i,j} =\sum _{l=1}^{d_i}Q_{i,j}^l$ where $Q_{i,j}^l$ are homogeneous polynomials of degree $l$. Now Lemma \ref{b} follows from Lemma \ref{analog}.
\end{proof}

%%%%%%%%%%%%%%%%%%%%%%%%%%%%%

\end{document}